\documentclass[11pt]{amsart}

\usepackage{palatino}
\usepackage{amsfonts}
\usepackage{amssymb}
\usepackage{amscd}
\usepackage{xypic}
\usepackage[breaklinks,bookmarksopen,bookmarksnumbered]{hyperref}
\usepackage{graphicx}
\usepackage{float}

\newcommand{\dblq}{{/\!/}}
\newtheorem{theorem}{Theorem}[section]
\newtheorem{proposition}[theorem]{Proposition}
\newtheorem{corollary}[theorem]{Corollary}

\theoremstyle{definition}
\newtheorem{definition}[theorem]{Definition}
\newtheorem{remark}[theorem]{Remark}

\newtheorem{conjecture/question}[theorem]{Conjecture/Question}

\newtheorem{remark/definition}[theorem]{Remark/Definition}
\newtheorem{terminology/notation}[theorem]{Terminology/Notation}

\def\GG{{\textbf G}}

\def\PP{{\textbf P}}

\def\VV{{\textbf V}}
\def\OO{\mathcal{O}}

\def\F{\mathcal{F}}
\def\P{\mathcal{P}}

\def\cS{\mathcal{S}}

\def\L{\mathcal{L}}
\def\C{\mathcal{C}}
\def\I{\mathcal{I}}

\def\cM{\mathcal{M}}
\def\cR{\mathcal{R}}

\def\rr{\overline{\mathcal{R}}}

\def\cU{\mathcal{U}}
\def\cC{\mathcal{C}}
\def\H{\mathcal{H}}

\def\mm{\overline{\mathcal{M}}}
\def\SS{\overline{\mathcal{S}}}
\def\hs{\mathcal{H}\mathcal{S}}

\def\th1{\overline{\Theta}_{g,1}^-}

\DeclareFontFamily{OT1}{pzc}{}
\DeclareFontShape{OT1}{pzc}{m}{it}{<-> s * [1.10] pzcmi7t}{}
\DeclareMathAlphabet{\mathpzc}{OT1}{pzc}{m}{it}

\begin{document}
\title{The unirationality of $\SS_9^{-}$ and moduli spaces of pointed spin curves}

\author[G. Farkas]{Gavril Farkas}

\address{Humboldt-Universit\"at zu Berlin, Institut f\"ur Mathematik,  Unter den Linden 6
\hfill \newline\texttt{}
 \indent 10099 Berlin, Germany} \email{{\tt farkas@math.hu-berlin.de}}

\author[A. Verra]{Alessandro Verra}
\address{Universit\`a Roma Tre, Dipartimento di Matematica e Fisica, Largo San Leonardo Murialdo \hfill
 \newline \indent 1-00146 Roma, Italy}
 \email{{\tt alessandro.verra@mat.uniroma3.it}}

\begin{abstract}
We show that the moduli space $\SS_9^-$of odd spin curves of genus $9$ is unirational. This is the highest genus for which such a result is known. This is achieved by realizing birationally the moduli space $\SS_g^-$ when $g\leq 9$ as a locally trivial projective bundle over a certain (finite quotient of the) moduli space $\SS_{g',n}^-$ of $n$-pointed odd stable spin curves of genus $g'<g$. We then present general results on the Kodaira dimension of both moduli spaces $\SS_{g,n}^-$ and $\SS_{g,n}^+$.  
\end{abstract}

\maketitle

\section{introduction}

The moduli space $\cS_g$ of spin curves of genus $g$ parametrizes pairs $[C, \vartheta]$ consisting of a smooth projective curve $C$ of genus $g$ and a theta-characteristic $\vartheta$, that is, a line bundle $\vartheta \in \mbox{Pic}^{g-1}(C)$ such  that $\vartheta^{\otimes 2}\cong \omega_C$. Using classical results \cite{Mu}, it is well-known that $\cS_g$ has two connected components $\cS_g^-$ and $\cS_g^{+}$ depending on the parity of  $h^0(C, \vartheta)$. Both components admit Deligne-Mumford compactifications $\SS_g^-$ and $\SS_g^+$ by means of stable spin curves which were constructed by Cornalba \cite{Cor}. The spin moduli spaces (as well as their pointed counterparts $\SS_{g,n}^-$ and $\SS_{g,n}^+$) have been intensely studied either from a geometric \cite{Lu}, tropical \cite{CMP}, or string-theoretic point of view \cite{DW}.  One has a complete classification by Kodaira dimension of both $\SS_g^-$ and $\SS_g^+$ for all genera, see \cite{F1}, \cite{FV1}, \cite{FV2}. This situation contrasts the case of the moduli space $\mm_g$, where in spite of the results in \cite{EH}, \cite{HM}, \cite{FJP}, the Kodaira dimension of $\mm_g$ when $17\leq g\leq 21$ remains unknown. 

\vskip 3pt

It is known that the odd spin moduli space $\SS_g^-$ is of general type for $g\geq 12$, uniruled for $g\leq 11$, and unirational for $g\leq 8$, see \cite{FV2}. The main result of this paper is the following:

\begin{theorem}\label{thm:main}
The odd spin moduli space $\SS_9^-$ is unirational.
\end{theorem} 

Note that this is the highest genus $g$ for which $\SS_g^-$ is known to be unirational. It has recently established that the Prym moduli space $\rr_9$ is uniruled \cite{FV3}. The even spin moduli space $\SS_9^+$ is of general type \cite{F2}; this is the lowest $g$ for which $\SS_g^+$ is of general type.  Therefore genus $9$ can be regarded as the transition case for both the Prym and spin moduli spaces.

\vskip 4pt

The proof of Theorem \ref{thm:main}  uses, on the one hand Mukai's celebrated structure theorem \cite{M2} for curves of genus $9$, on the other hand a fundamental construction in \cite{FV2} applying to all odd spin moduli spaces $\SS_g^-$ for $g\leq 9$. Precisely, denoting by 
$$\mathfrak{V}:=\mbox{Sp}(3,6)\subseteq \PP^{13}$$
the $6$-dimensional symplectic Grassmannian, since $K_{\mathfrak{V}}=\OO_{\mathfrak{V}}(-4)$, linear $1$-dimensional sections of $\mathfrak{V}$ are canonical curves of genus $9$. It is the main result of \cite{M2} that \emph{every} non-pentagonal curve $[C]\in \cM_9$ appears in this way, that is, $C=\mathfrak{V} \cap H_1\cap \ldots \cap H_5\subseteq \PP^8$, where $H_i\in |\OO_{\PP^{13}}(1)|$ are hyperplanes. Following \cite{FV2} we denote by $\Xi \subseteq \mbox{Hilb}_{16}(\mathfrak{V})$ the variety of length $16$ \emph{clusters}, consisting of $0$-dimensional curvilinear subschemes $Z\in \mbox{Hilb}_{16}(\mathfrak{V})$ with $\mbox{supp}(Z)=\{p_1, \ldots, p_8\}$ such that $\mbox{mult}_{p_i}(Z)=2$, for $i=1, \ldots, 8$. Introducing the  variety
$$\cU:=\Bigl\{(C, Z): C=H_1\cap \cdots \cap H_5\cap \mathfrak{V} \mbox{ is a smooth curve section},\  \ Z\subseteq C\Bigr\}, $$
and retaining the previous notation, 
one has an $\mbox{Aut}(\mathfrak{V})$-equivariant map $\cU\dashrightarrow \SS_9^-$ given by $\bigl(C,Z)\mapsto [C, \OO_C(p_1+\cdots+p_8)]$. It is shown in \cite[Proposition 3.7]{FV2} that this construction induces a birational isomorphism 
\begin{equation}\label{eq:Mukaimodel_9}
\mathfrak{s}_9:=\cU \dblq \mbox{Aut}(\mathfrak{V})\stackrel{\cong}\dashrightarrow \SS_9^{-}. 
\end{equation}
We regard $\mathfrak{s}_9$ as being the \emph{Mukai model} of $\SS_9^-$. For every $1\leq \delta \leq g$, we denote by $\cU_{\delta}$ the subvariety of $\cU$ consisting of pairs $(\Gamma, Z)$ such that $\Gamma$ is a $\delta$-nodal curve with $\mbox{sing}(\Gamma)\subseteq \mbox{supp}(Z)$. We write $\mbox{supp}(Z)=\{p_1, \ldots, p_8\}$ and we may assume that $\{p_1, \ldots, p_{\delta}\}=\mbox{sing}(\Gamma)$. If $\nu\colon N\rightarrow \Gamma$ is the normalization of $\Gamma$ and $\{x_i, y_i\}=\nu^{-1}(p_i)$ for $i=1, \ldots, \delta$, then 
$$\bigl[N, x_1+y_1, \ldots, x_{\delta}+y_{\delta}, \ \OO_N(p_{\delta+1}+\cdots+p_8)\bigr]$$ can be regarded as an odd $2\delta$-pointed spin curve of genus $9-\delta$. Letting $B_{g, \delta}^-:=\SS_{g-\delta,2\delta}^-/\mathbb Z_2^{\oplus \delta}$ to be the moduli space of odd spin curves of genus $g-\delta$ with $\delta $ pairs of points, one has a birational isomorphism 
\begin{equation}\label{eq:nodalMukaimodel_9}
\mathfrak{s}_{9, \delta}:=\cU_{\delta}\dblq \mbox{Aut}(\mathfrak{V})\stackrel{\cong}\dashrightarrow B_{9,\delta}^-.
\end{equation}
One can combine the isomorphisms (\ref{eq:Mukaimodel_9}) and (\ref{eq:nodalMukaimodel_9}) into the following incidence correspondence 
\begin{equation}\label{eq:incidence2}
\xymatrix{
  & \mathcal{P}_{9, \delta}:=\Bigl\{(C,\Gamma, Z)\in \cU \times_{\Xi} \cU_{\delta}: Z\subseteq C\cap \Gamma\Bigr\}\dblq \mbox{Aut}(\mathfrak{V})
   \ar[dl]_{\bar{\alpha}} \ar[dr]^{\bar{\beta}} & \\
   \SS_9^{-} & & B_{9,\delta}^-       \\
                 }
\end{equation}                 
where the maps $\bar{\alpha}$ and $\bar{\beta}$ are induced by the birational isomorphisms (\ref{eq:Mukaimodel_9}) and (\ref{eq:nodalMukaimodel_9}) respectively. Note that $\bar{\beta}$ is birational to a Zariski locally trivial projective bundle, whereas $\bar{\alpha}$ is surjective if and only if $\delta\leq \mbox{dim}(\mathfrak{V})-1$, see \cite[Proposition 3.13 ]{FV2}. This construction has been used in \cite[Theorem 0.2]{FV2} to show that $\SS_g^{-}$ is unirational for $g\leq 8$ by taking $\delta=g-1$, in which case $B_{g,g-1}^-$ is a moduli space of pointed elliptic curves. Due to the small dimension of the Mukai variety $\mathfrak{V}$, this strategy does not work for the moduli space $\SS_9^-$, therefore this case remained open. It is one of the main tasks of this paper to complete this program in genus $9$. We take $\delta=\mbox{dim}(\mathfrak{V})-1$ and then Theorem \ref{thm:main} will follow from the following result:

\begin{theorem}\label{thm:kodaira4}
The moduli space $\SS_{4,n}^{-}$ of $n$-pointed odd spin curves of genus $4$ is unirational for $n\leq 10$.
\end{theorem} 

One has a finite surjective map $\SS_{4,10}^-\twoheadrightarrow B_{9,5}^{-}$. Therefore,  from the correspondence (\ref{eq:incidence2}) applied in the case $\delta=5=\mbox{dim}(\mathfrak{V})-1$, we obtain via Theorem \ref{thm:kodaira4}  that $\mathcal{P}_{9,5}$ is unirational. Since $\bar{\alpha}$ is a dominant morphism, we conclude that the spin moduli space $\SS_9^-$ is unirational. It is an interesting open question in the spirit of \cite{CL} whether the Chow wing $CH^*(\SS_g^-)$ is tautological for $g\leq 9$.

\vskip 4pt

We now discuss the main ideas in the proof of Theorem \ref{thm:kodaira4}. We start with a general triple $[C,\vartheta, p]$, where $[C, \vartheta]\in \cS_4^-$ and $p\in \mathrm{supp}(\vartheta)$. Observe that such a triple induces a plane quintic model $\Gamma\subseteq \PP^2$ with a fixed bitangent line. Precisely, the canonical curve $C\subseteq \PP^3$ lies on a unique quadric $Q\subseteq \PP^3$ and we denote by $\ell'$ and $\ell''$ the two rulings of $Q$ passing through $p\in C$, therefore ${\bf T_p}(Q)=\langle \ell', \ell''\rangle\subseteq \PP^3$ is the tangent plane to the quadric $Q$. Then if 
$$\phi=\phi_{|\omega_C(-p)|} \colon C \longrightarrow \PP^2=\PP H^0\bigl(\omega_C(-p)\bigr)^{\vee}$$
is the corresponding projection, we set $n:=\phi(p)$ and denote by $n':=\phi(\ell')$ and $n'':=\phi(\ell'')$ the two nodes of the projected canonical curve $\Gamma:=\phi(C)$. We set $F:=\langle n', n''\rangle\subseteq \PP^2$. Writing 
$$\mbox{supp}(\vartheta)=p+t_1+t_2,$$
we consider a second line $L=\langle \phi(t_1), \phi(t_2)\rangle\subseteq \PP^2$. The lines $F$ and $L$ meet in the point $n$ and the curve $\Gamma$ is tangent to $L$ at both points $\phi(t_1)$ and $\phi(t_2)$. Observe that the union of the two lines $L+F$ can be viewed as a (degenerate) totally tangent conic to the quintic curve $\Gamma$. 

\vskip 4pt

We now reverse this construction and start with two fixed lines $F$ and $L$ in $\PP^2$. We denote by $\{n\}=F\cap L$ and fix two further general points $n', n'' \in F$. We let
\begin{equation}\label{eq:P13}
\PP=\PP^{13}=\bigl|\OO_{\PP^2}(5)(-n-2n'-2n'')\bigr|\subseteq \bigl|\OO_{\PP^2}(5)\bigr|
\end{equation}
be the linear system of quintics nodal at $n'$ and $n''$ and passing through $n$. One has a map 
\begin{equation}\label{eq:rho}
\rho\colon \PP\dashrightarrow L^{[4]}\cong \PP^4
\end{equation}
given by assigning to each quintic curve $\Gamma\in \PP$ the cycle $\Gamma\cdot L\setminus \{n\}$. Note that we identify the $a$-th symmetric product $L^{[a]}$ of $L$ with the projective space $\PP^a$.

\vskip 4pt

The key observation is, that under this identification, the image of the \emph{squaring map}
$$L^{[2]}\longrightarrow L^{[4]}, \ \ x+y\mapsto 2(x+y),$$
can be identified with the \emph{projected Veronese surface} $\VV \subseteq \PP^4$. In this way the cone $\rho^*(\VV)\subseteq \PP$ over $\VV$ can be regarded as the family of plane quintics passing through $n$, nodal at $n'$ and $n''$, which are bitangent to the line $L$ at two unspecified points.

\vskip 4pt

In order to parametrize $\SS_{4,10}^-$, we start with a general point $\bar{p}=(p_1, \ldots, p_{10})\in (\PP^2)^{10}$ and let $\PP^3_{\overline{p}}=\bigl|\OO_{\PP^2}(5)(-n-2n'-2n''-p_1-\cdots-p_{10})\bigr|$ be the $3$-dimensional space of quintic curves  $\Gamma\in \PP$ passing through the points in $\overline{p}$. We then have a rational quartic curve 
\begin{equation}\label{eq:re}
R_{\overline{p}}:=\PP^3_{\overline{p}}\cdot \rho^*(\VV)\subseteq \PP^3_{\overline{p}}
\end{equation}
parametrizing those quintics in $\PP^3_{\overline{p}}$ which are bitangent to $L$. One has two fibrations 
\begin{equation}\label{eq:quartic_fibr}
\xymatrix{
  & \cR:=\Bigl\{(\overline{p}, \Gamma): \overline{p}=(p_1, \ldots, p_{10})\in (\PP^2)^{10}, \ \ \ \Gamma\in R_{\overline{p}} \Bigr\}
   \ar[dl]_{m} \ar[dr]^{\mathfrak{q}} & \\
   \SS_{4,10}^{-} & & \bigl(\PP^2\bigr)^{10}       \\
                 }
\end{equation}
where $\mathfrak{q}^{-1}(\overline{p})=R_{\overline{p}}$, whereas $m$ is the moduli map assigning to $\bigl(\overline{p}, \Gamma\bigr)$ the pointed spin curve 
$$\bigl[C=\phi^{-1}(\Gamma), \vartheta=\OO_C(p+t_1+t_2), \ p_1, \ldots, p_{10}\bigr]\in \SS_{4,10}^-,$$ 
where $\phi\colon C\rightarrow \Gamma$ denotes the normalization map, $p=\phi^{-1}(n)$ and $L\cdot \Gamma=n+2\bigl(\phi(t_1)+\phi(t_2)\bigr)$. Observe that  $\mathfrak{q}$ is a birational to a fibration in rational (quartic curves), whereas $m$ is dominant. Now Theorem \ref{thm:kodaira4} follows by constructing a \emph{rational bisection} of $\mathfrak{q}$ which is a rational variety. Via a standard base change argument, this immediately implies that $\cR$, and therefore also $\SS_{4,10}^-$, are unirational, which finishes the proof.

\vskip 5pt

\noindent {\bf{The Kodaira dimension of $\SS_{g,n}$.}} Theorem \ref{thm:kodaira4} presented above invites the more general question which moduli spaces $\SS_{g,n}^{\pm}$ of $n$-pointed spin curves of genus $g$ are of general type. For similar results in the much more studied case of $\mm_{g,n}$, we refer to \cite{Log}, \cite{F1} and \cite{AB}. Note also the progress \cite{B}, \cite{Bu}, \cite{CCM} on determining the Kodaira dimension of the strata of abelian differentials,  which can be thought of as  generalizations of the spaces $\SS_{g,n}^-$. Introducing the quotient $\SS_{g,[n]}:=\SS_{g,n}/\mathfrak{S}_n$, we have an isomorphism 
$$B_0\cong \SS_{g-1,[2]},$$
where $B_0$ is the ramification divisor of the map $\SS_g\rightarrow \mm_g$. We prove the following result:

\begin{theorem}\label{thm:kodaira_pointed}
1) The even pointed spin moduli space $\SS_{g,n}^+$ is of general type for $n\geq f(g)$:
\begin{table}[htp!]
%\begin{center}
\begin{tabular}{|c|c|c|c|c|c|c|c|c|}
\hline
$g$ & $4$ & $5$ & $6$ & $7$ & $8$ & $\geq 9$ \\
\hline
$f(g)$ & $9$ & $7$ & $7$ & $4$ & $1$  & $0$\\
\hline
\end{tabular}
%\end{center}
\end{table}

\noindent Furthermore, $\SS_{7,3}^+$ has non-negative Kodaira dimension.

\noindent 2) The odd spin pointed moduli space $\SS_{g,n}^-$ is of general type for $n\geq h(g)$:
\begin{table}[htp!]
%\begin{center}
\begin{tabular}{|c|c|c|c|c|c|c|c|c|}
\hline
$g$ & $4$ & $5$ & $6$ & $7$ &  $8$  & $9$ & $10$  \\
\hline
$h(g)$ & $12$ & $11$ & $10$ & $7$ & $5$ & $4$ & $2$ \\
\hline
\end{tabular}
\end{table}

\noindent Furthermore, both spaces $\SS_{8,4}^-$ and $\SS_{9,3}^-$ have non-negative Kodaira dimension.
\end{theorem}

It is an interesting open question whether the moduli spaces $\SS_{7,3}^+$, $\SS_{8,4}^-$ and $\SS_{9,3}^-$ are indeed of Kodaira dimension zero, and if so, to find Calabi-Yau models for them. The proof of Theorem \ref{thm:kodaira_pointed} consists of two parts. First, by using Ludwig's results \cite{Lu}, we show that as long as $g\geq 4$, any pluricanonical form on the coarse moduli spaces $\SS_{g,n}^{\pm}$ extends to any resolution of singularities. Then  we show  that the canonical class $K_{\SS_{g,n}^+}$ (respectively $K_{\SS_{g,n}^-}$) is big for $n\geq f(g)$ (respectively for $n\geq h(g)$). To that end we use our extensive knowledge of the cones of effective divisors of $\mm_{g,n}$ provided by \cite{Log}, \cite{F1}, respectively of the unpointed spin moduli spaces $\SS_g^{\pm}$. An important new divisor class calculation provided in this paper is that of the \emph{universal odd spin structure}, that is, the closure $\overline{\Theta}_{g,1}$ in $\SS_{g,1}^-$ of the locus
$$\Theta_{g,1}:=\Bigl\{[C, \vartheta, p]\in \cS_{g,1}^-: H^0\bigl(C, \vartheta(-p)\bigr)\neq 0\Bigr\}$$
of points in the support of odd spin structures. The class of $\overline{\Theta}_{g,1}$ is computed in Theorem \ref{thm:univ_theta}.

Particularly interesting is the case $g=11$. It is shown in \cite{FV2} that $\SS_{11}^-$ is uniruled, this being the highest $g$ for hich $\SS_g^-$ is not of general type. We have the following result:

\begin{theorem}\label{thm:gen11}
\hfill

1) The moduli space $\SS_{11,1}^-$ is not of general type; it has Kodaira dimension at least $19$.

2) The moduli space $\SS_{11,[2]}^-$ has Kodaira dimension at least $19$.
\end{theorem}

We show that $K_{\SS_{11,1}^-}$ is a linear combination of $\bigl[\overline{\Theta}_{11,1}^-\bigr]$, the pull-back of the Brill-Noether divisor under the forgetful map $\SS_{11,1}^-\rightarrow \mm_{11}$ and the pull-back under the forgetful map $\SS_{11,1}^-\rightarrow \SS_{11}^-$ of the branch divisor of the generically finite map $\overline{\Theta}_{11,1}^-\rightarrow \SS_{11}^-$. Furthermore, we show in Theorem \ref{thm:uniruledK3} that through a general point of $\overline{\Theta}_{11,1}^-$ there passes a rational curve $\Gamma\subseteq \SS_{11,1}^-$ such that $\Gamma \cdot K_{\SS_{11,1}^-}=0$, which immediately implies (the first part of) Theorem \ref{thm:gen11}. We are tempted to conjecture that both $\SS_{11,1}^-$ and $\SS_{11,[2]}^-$ have Kodaira dimension $19$, this being the dimension of the moduli space $\F_{11}$ of polarized $K3$ surfaces of genus $11$. Note that it has been shown in \cite{FV4} that the universal Jacobian over $\mm_{11}$ has Kodaira dimension $19$, though we cannot see a direct connection to $\SS_{11,1}^-$. There are other interesting moduli spaces of curves of genus $10$ and $11$ that have submaximal Kodaira dimension \cite{BM}.  

\vskip 5pt

The last topic we mention concerns the moduli space $\SS_{1,n}^+$. It is a well known result of Deligne that $H^{11,0}(\mm_{1,11})\neq 0$, see \cite[\S 3.5]{FP}, or \cite{BF} for various presentations. This has been one of the sources for producing non-tautological cohomology classes on $\mm_{g,n}$, see \cite{CLP} and references therein. We have the following result\footnote{This contradicts \cite[Lemma 2]{BF}, which is incorrect.} on the space $\SS_{1,n}^+$:
\begin{theorem}\label{thm:genus1}
One has $H^{7,0}(\SS_{1,7}^+)\neq 0$. Furthermore, the Kodaira dimension of $\SS_{1,7}^+$ is equal to zero, whereas the Kodaira dimension of $\SS_{1,n}^+$ is equal to one for $n\geq 8$.
\end{theorem}

Note that Krug \cite[Corollary 4.25]{Kr} showed that $\SS_{1,n}^+$ is rational for $n\leq 6$. Coupled with Theorem \ref{thm:genus1}, one has a complete  description of the Kodaira dimension of $\SS_{1,n}^+$ for every $n$.

\vskip 6pt

\noindent {\small {\bf Acknowledgments:} Farkas was partly supported by the Berlin Mathematics Research Center MATH+ and by the ERC Advanced Grant SYZYGY (grant agreement No. 834172). This work was started when Verra was an invited professor of the Berlin Mathematical School. }

\section{Spin curves of genus 4 and the projected Veronese surface} 

We begin by recalling some of the basic properties of the projected Veronese surface, whic are essential for the proof of Theorem \ref{thm:kodaira4}. We let $L\cong \PP^1$ be the smooth rational curve and denote by $L^{[n]}\cong \mbox{Hilb}^n(L)$ its $n$-th symmetric product. We consider the Veronese embedding 
$$\nu\colon \PP^2\cong \bigl|\OO_{\PP^2}(1)\bigr|\longrightarrow \bigl|\OO_{\PP^2}(2)\bigr|\cong \PP^5, \ \ \  [v]\mapsto [v^2],$$ 
identifying the Veronese surface $\nu(\PP^2)$ with the space of rank one conics in $\PP^2$. We then introduce the \emph{squaring map}
\begin{equation}\label{eq:square}
\mbox{sq}\colon L^{[2]}\longrightarrow L^{[4]}, \  \ \ \mbox{sq}(x+y):=2(x+y).
\end{equation}
Under the identification $L^{[2]}\cong \PP^2$,  the map $\mbox{sq}$ can be given in coordinates as 
$$[a,b,c]\mapsto \bigl[a^2, b^2, c^2+2ab, 2ac,2bc\bigr]\in \PP^4.$$
In particular, the smooth surface $\VV:=\mbox{Im}(\mbox{sq})\subseteq L^{[4]}\cong \PP^4$ can be identified with an isomorphic projection of the Veronese surface $\nu(\PP^2)$. Severi showed that $\VV\subseteq \PP^4$ is the only smooth non-degenerate surface that can be obtained as a linear projection.

\vskip 4pt

It is a classical result \cite[Theorem 2.1.4]{Dol} that the variety ${\bf{D}}:=\bigl\{\ell\in \GG(1,4): \ell \cdot \VV \geq 3\bigr\}^{-}$ representing the closure in the Grassmannian of lines $\GG(1,4)\subseteq \PP^9$ of the family of trisecant lines to $\VV$ is a de Pezzo threefold being a smooth linear section of $\GG(1,4)$. Furthermore, the universal line over $\bf{D}$, that is, 
$$\xymatrix{
  & \mathcal{X}:=\Bigl\{(x,\ell): \ell\in {\bf{D}},\  x\in \ell\subseteq \PP^4 \Bigr\}\subseteq \PP^4\times {\bf{G}}(1,4)
   \ar[dl]_{\pi_1} \ar[dr]^{\pi_2} & \\
   \PP^4 & & \bf{D}       \\
                 }
$$
realizes $\pi_1$ as a birational map, that is, for a general point $o\in \PP^4$ there is a unique trisecant line $\ell\in \bf{D}$ passing through $o$.

\subsection{Plane quintic models of odd spin curves of genus $4$.} Let $[C, \vartheta, p]\in \Theta_{4,1}^-$ be a general point, consisting of a theta characteristic $\vartheta$ with $h^0(C, \vartheta)=1$ and write $\mbox{supp}(\vartheta)=p+t_1+t_2$. We denote by $A', A''\in W^1_3(C)$ the two minimal 
pencils on $C$, thus $A'\otimes A''=\omega_C$.

\begin{proposition}\label{prop:transv}
For a general point $[C, \vartheta, p]\in \Theta_{4,1}^-$, for both $A', A''\in W^1_3(C)$, we have 
$$H^0(C, A'(-2p))=0, \ H^0(C, A''(-2p))=0.$$
Furthermore, both divisors $A'(-p)$ and $A''(-p)$ are reduced.
\end{proposition}
\begin{proof}
We use that $\Theta_{4,1}^-$ is irreducible, see also Remark \ref{rmk:irr}. Assuming that for a general point $[C, \vartheta, p]\in \Theta_{4,1}^-$, the point $p$ is either a ramification, or an antiramification point{\footnote{We say that $p\in C$ is an antiramification point of a pencil $f\colon C\rightarrow \PP^1$ if $f^{-1}(p)$ contains a ramification point of $f$, or equivalently, $f(p)\in \PP^1$ is a branch point.}} of a pencil $A'\in W^1_3(C)$, it follows that  $A', A''\in W^1_3(C)$ have in total at least $(g-1)\cdot 2^{g-1}(2^g+1)=408$ ramification or antiramification points. But the total number of ramification and antiramification points of $A'$ and $A''$ equals $4\cdot \bigl(2g+2\cdot \mbox{deg}(A')-2\bigr)=56$, a contradiction.   
\end{proof}

We denote by $\phi=\phi_{|\omega_C(-p)|}\colon C\rightarrow \PP^2$ the projected canonical curve and by $\Gamma:=\mbox{Im}(\phi)$ the corresponding quintic plane model. We write  
$$A'=\OO_C(p+n_1'+n_2') \ \mbox{ and } A''=\OO_C(p+n_1''+n_2''),$$
where by Proposition \ref{prop:transv} the points $n_1', n_2', n_1'', n_2''$ and $p$ are pairwise distinct. It follows that the curve $\Gamma$ has two nodes at the points $n'=\phi(n_1')=\phi(n_2')$ respectively $n''=\phi(n_1'')=\phi(n_2'')$. As explained in the Introduction, we denote by $$F=\bigl\langle n', n''\bigr \rangle \subseteq \PP^2$$ the line spanned by the nodes of $\Gamma$. Set also $n:=\phi(p)$. Since $n_1'+n_2'+n_1''+n_2''+p\in \bigl|\omega_C(-p)\bigr|$, it follows that $n\in F$.

\vskip 3pt

We  set $b_1:=\phi(t_1)$ and $b_2:=\phi(t_2)$ and introduce the line $L:=\bigl\langle b_1,b_2\bigr\rangle\subseteq \PP^2$. Note that
$$\Gamma\cdot L=\phi_*(C)\cdot L=\phi_*(2t_1+2t_2+p)=2b_1+2b_2+n,$$
showing, on the one hand that $n\in L$, on the other that $\Gamma$ is bitangent to $L$ at the points $b_1$ and $b_2$. The data contained in  $L$ and $F$ is that of a rank $1$ conic which is totally tangent to the quintic curve $\Gamma$, that is,
\begin{equation}\label{eq:deg_conic}
(L+F)\cdot \Gamma=2\bigl(n'+n''+n+b_1+b_2\bigr).
\end{equation}

\vskip 5pt
 
As explained in the Introduction, in order to parametrize the moduli space of odd spin curves of genus $4$, we now turn this construction around and fix a line $F\subseteq \PP^2$ and distinct points $n', n'', n\in F$. We also fix a general line $L\subseteq \PP^2$ passing through $n$. Let 
$$S:=\mbox{Bl}_{\{n,n',n''\}}(\PP^2)\stackrel{\epsilon}\longrightarrow \PP^2$$ be the blow-up of $\PP^2$ at $n, n'$ and $n''$  and denote by $E,E', E''$ the exceptional divisors over the points $n, n'$ respectively $n''$. Let $L'\subseteq S$ be the proper transform of  $L$.
 
 \vskip 4pt
 
We introduce the linear system of $2$-nodal plane quintics $\PP:=\bigl|5h-E-2E'-2E''\bigr|$, where $h\in |\epsilon^*\OO_{\PP^2}(1)|$ and consider the map $\rho\colon \PP\dashrightarrow L^{[4]}$ described by (\ref{eq:rho}) and which is the projectivization of the surjective map $\rho$ in the following exact sequence:
\begin{equation}\label{eq:resl}
0\longrightarrow H^0\bigl(S, \OO_{S}(4h-2E'-2E'')\bigr)\stackrel{\cdot L'}\longrightarrow H^0\bigl(S, \OO_S(5h-2E'-2E''-E)\bigr)\stackrel{\rho}\longrightarrow H^0\bigl(L', \OO_{L'}(4)\bigr)\longrightarrow 0.
\end{equation}
In this way, the map $\rho\colon \PP\dashrightarrow L^{[4]}$ is a linear projection.

\begin{definition}
Let $\rho\colon \cC:=\rho^*(\VV)\rightarrow \VV$ denote the cone over the projected Veronese surface $\VV$.
\end{definition} 
Clearly $\cC$ can be identified with the parameter space of those plane quintics $\Gamma\in \PP$ which are bitangent to the line $L$. Note that $\cC$ being a cone over a rational surface, is a rational $11$-dimensional variety.

\begin{proposition}\label{prop:dom2}
One has a dominant map $m^- \colon \cC\twoheadrightarrow \SS_4^-$.
\end{proposition}
\begin{proof}
We take a general point $[\Gamma]\in \PP$, where the plane quintic $\Gamma\subseteq \PP^2$ satisfies the relation (\ref{eq:deg_conic}). Retaining the same notation, we may assume the points $n', n''$, $b_1$ and $b_2$ are pairwise distinct and we assume $\Gamma$ is tangent to $L$ at the points $b_1$ and $b_2$. Let
$$\nu\colon C\longrightarrow \Gamma$$ be the normalization map and set $p:=\nu^{-1}(n)$, where $\{n\}=L\cap F$. Let $t_i:=\nu^{-1}(b_i)$, for $i=1, 2$. Setting $\vartheta=\OO_C(p+t_1+t_2)$, from (\ref{eq:deg_conic}), we obtain that $\vartheta^{2}\cong \omega_C$, therefore $[C, \vartheta, p]\in \Theta_{4,1}^-$. We then define $m^{-}\bigl([\Gamma]\bigr)=[C, \vartheta]$. Our discussion above shows this map is dominant.
\end{proof}

\subsection{The unirational parametrization of $\SS_{4,10}^-$.} Let us now fix a general point 
$$\overline{p}=(p_1, \ldots, p_{10})\in \bigl(\PP^2\bigr)^{10},$$
and we do not distinguish whether we regard these points in $\PP^2$, or in $S$. We  denote by $\PP^3_{\overline{p}}:=\bigl\{\Gamma \in \PP: p_1, \ldots, p_{10}\in \Gamma\bigr\}=\bigl|\I_{\{p_1, \ldots, p_{10}\}/S}(5h-E-2E'-2E'')\bigr|$. We consider the intersection $R_{\overline{p}}:=\cC\cdot \PP^3_{\overline{p}}$  already introduced in (\ref{eq:re}). We have the following:

\begin{proposition}\label{prop:quartic1}
If ${\overline{p}}\in (\PP^2)^{10}$ is general, then $R_{\overline{p}}\subseteq \PP^3_{\overline{p}}$ is a smooth quartic rational curve.
\end{proposition}
\begin{proof}
We claim that when $\overline{p}$ is general as above, then $\PP^3_{\overline{p}}\cong \rho_*\bigl(\PP^3_{\overline{p}}\bigr)$ is a general hyperplane in $L^{[4]}\cong \PP^4$. Indeed, we use the sequence (\ref{eq:resl}) and observe that by the generality assumption of the points $p_1, \ldots, p_{10}$, we have
$$H^0\bigl(S, \I_{\{p_1,\ldots, p_{10}\}/S}(4h-2E'-2E'')\bigr)=0,$$
therefore $\rho_*\bigl(\PP^3_{\overline{p}}\bigr)$ is a hyperplane in $L^{[4]}$. As $\overline{p}$ varies in $\bigl(\PP^2\bigr)^{10}$, then $R_{\overline{p}}$ is a general hyperplane section of $\VV$. By Bertini's theorem, it is therefore a smooth rational quartic curve.
\end{proof} 

We now globalize this construction and ultimately obtain  a variety dominating $\SS_{4,10}^-$. We denote by $\cU\subseteq \bigl(\PP^2\bigr)^{10}$ the dense open subset of configurations $\overline{p}\in\bigl(\PP^2\bigr)^{10}$ such that $\mbox{dim } \bigl|\I_{\{p_1, \ldots, p_{10}\}/S}(5h-E-2E'-2E'')\bigr|=3$.
We then consider the projective bundle
$$q\colon \P\longrightarrow \cU, \ \mbox{ with } \ q^{-1}(\overline{p})=\PP^3_{\overline{p}}, \mbox{ for each } \overline{p}\in \cU.$$
We introduce the subvariety $\cR=\bigl\{(\overline{p}, \Gamma):\overline{p}\in \cU, \ \Gamma\in R_{\overline{p}}\bigr\}\subseteq \P$. Using Proposition \ref{prop:quartic1}, we conclude that the projection 
\begin{equation}\label{eq:fibr_q}
\mathfrak{q} \colon \cR \longrightarrow \cU
\end{equation}
is birationally a fibration in rational quartic curves over the rational variety $\cU$. In particular, using \cite{GHS}, we obtain that  $\cR$ is rationally connected, though we shall soon prove that it is in fact unirational.

\vskip 3pt

There is a moduli map $m\colon \cR\dashrightarrow \SS_{4,10}^-$, which we now describe. Since there is a dominant map $\cR\twoheadrightarrow \cC$, a general point $(\overline{p}, \Gamma)\in \cR$ corresponds to an irreducible  plane quintic $\Gamma\subseteq \PP^2$, which is nodal at $n'$ and $n''$ and satisfies the equations, cf. (\ref{eq:deg_conic}): 
$$\Gamma \cdot F=2n'+2n''+n \ \   \mbox{ and } \ \   \Gamma\cdot L=2b_1+2b_2+n.$$
Choosing the points $p_1, \ldots, p_{10}$ generally, we may also assume they are smooth points of $\Gamma$.
As in the proof of Proposition \ref{prop:dom2}, we denote by $\nu\colon C\rightarrow \Gamma$ the normalization map and set $p=\nu^{-1}(n)$, $t_1=\nu^{-1}(b_1)$ and $t_2=\nu^{-1}(b_2)$. We then define the map
$$m\bigl([\overline{p}, \Gamma]\bigr)=\bigl[C, \OO_C(p+t_1+t_2),\  p_1, \ldots, p_{10}\bigr]\in\SS_{4,10}^-,$$
where we identify $p_i$ and $\nu^{-1}(p_i)$. We  summarize what has been achieved so far:

\begin{proposition}\label{prop:dom3}
The moduli map $m\colon \cR\longrightarrow \SS_{4,10}^-$ is dominant.
\end{proposition}
\begin{proof} Follows immediately by combining Propositions \ref{prop:transv} and \ref{prop:dom2}.    
\end{proof}

\vskip 3pt

The canonical way to show that $\cR$ is unirational, would be by constructing a rational section of the fibration $\mathfrak{q} \colon \cR\rightarrow \cU$ defined in (\ref{eq:fibr_q}).  It is however unclear whether such a section exists. Instead, we first construct a natural unirational \emph{bisection} of $\mathfrak{q}$. To that end, we fix a conic 
$$Q\subseteq \VV\subseteq \PP^4.$$
Note that $\VV$ has a $2$-dimensional family of such conics which are image of lines in $\PP^2$ under the map $\mathrm{sq}$ described in (\ref{eq:square}).
We denote by $\mathrm{Hilb}^2(\mathfrak{q})\rightarrow \cU$ the relative Hilbert scheme of length $2$ subschemes of the fibration $\mathfrak{q}$.
\begin{definition}
The rational section $\sigma\colon \bigl(\PP^2\bigr)^{10}\dashrightarrow \mathrm{Hilb}^2(\mathfrak{q})$ is defined by setting 
$\sigma(\overline{p}):=R_{\overline{p}}\cdot Q$.
\end{definition}
Note that $\sigma(\overline{p})$ corresponds to the intersection in $L^{[2]}$ of a conic corresponding to $R_{\overline{p}}$ with the fixed line whose image under $\mathrm{sq}$ defines the conic $Q$, therefore $\sigma(\overline{p})$ is indeed a length $2$ subscheme of $R_{\overline{p}}$. Observe that $\sigma$ also gives rise to a rational morphism 
\begin{equation}\label{map:f}
f\colon \cU\dashrightarrow Q^{[2]},  \ \ \overline{p}\mapsto R_{\overline{p}}\cdot Q.
\end{equation}
\begin{proposition}\label{prop:f}
The fibration $f$ is birational to a Zariski locally trivial ${\bf{G}}(1,10)$-bundle over $Q^{[2]}$.
\end{proposition}
\begin{proof}
We choose points $x, x'\in Q$ and let $\ell=\langle x,x'\rangle\subseteq \langle Q\rangle$ be the line spanned by $x$ and $x'$ in $\PP^4$. Then $f^*(x+x')$ corresponds to those points $\overline{p}\in \cU$ such that $x, x'\in \rho(\PP^3_{\overline{p}})$. Identifying $x, x'\in \VV$ with the divisors $2(a+b)$ respectively $2(a'+b')$, where $a,a', b,b'\in L$, then $f^*(x+x')$ corresponds to those points $\overline{p}$ for which there exists nodal quintics $\Gamma, \Gamma'\in \PP^3_{\overline{p}}$ such that 
$$\Gamma\cdot L=n+2a+2b, \ \ \mbox{ and } \ \ \Gamma\cdot L'=n+2a'+2b'.$$
Then recalling that the map $\rho$ defined in (\ref{eq:rho}) is a linear projection, this immediately implies that $\mbox{dim } \PP^3_{\overline{p}}\cdot \rho^*\bigl(\ell)\geq 1$. Conversely, if the intersection $\rho^*(\ell)\cdot \PP^3_{\overline{p}}$ is positive dimensional, there exist two points $x,x'\in \ell\cdot Q\subseteq \langle Q\rangle$ such that $f(\overline{p})=x+x'$. We conclude that the general fibre $f^*(x+x')$ is isomorphic to the Grassmannian of lines ${\bf{G}}(1,\rho^*\ell)$. Varying now $x+x'\in Q^{[2]}$, we obtain that $f$ is birational to a locally trivial ${\bf{G}}(1,10)$-bundle over $Q^{[2]}$. 
\end{proof}

We now denote by $\tau\colon Q\times Q\rightarrow Q^{[2]}$ the double cover given by $\tau(x,y)=x+y$ and set 
$$\tilde{f}\colon \widetilde{\cU}:=\cU\times_{Q^{[2]}} \bigl(Q\times Q\bigr)\longrightarrow Q\times Q.$$
Therefore $\widetilde{\cU}$ is the parameter space of pairs $(\overline{p},x)$, where $\overline{p}\in \cU$ and $x\in \PP^3_{\overline{p}}\cdot Q$. 
\begin{corollary}\label{cor:deg2}
The morphism $\tilde{f}\colon \widetilde{\cU}\rightarrow Q\times Q$ is birational to a locally trivial $\GG(1,10)$-bundle. In particular,  $\widetilde{\cU}$ is a rational variety.
\end{corollary} 
\begin{proof}
Follows immediately by base change from Proposition \ref{prop:f}.
\end{proof}

We now base change the fibration in rational curves $\mathfrak{q}$, to obtain the family
$$\tilde{\mathfrak{q}}\colon \widetilde{\cR}=\cR\times _{\cU} \widetilde{\cU}\longrightarrow \widetilde{\cU}.$$
The general fibre of $\tilde{\mathfrak{q}}$ is still a rational quartic curve. Note that there exists a rational section $\tilde{\sigma}\colon \widetilde{\cU}\dashrightarrow \widetilde{\cR}$ given by $\tilde{\mathfrak{q}}\bigl(\overline{p}, x):=\bigl(\overline{p}, R_{\overline{p}}, \Gamma_x\bigr).$
Here, $\Gamma_x\in R_{\overline{p}}$ denotes the nodal quintic curve corresponding to the point $x\in Q\cdot \PP^3_{\overline{p}}$.
The situation is summarized by the following diagram:   

$$\xymatrix@R=3pc@C=4pc{
\widetilde{\cR}\ar[d]_{\tilde{\mathfrak{q}}}  \ar[r]& \cR \ar[d]_{\mathfrak{q}} \ar@{->>}[dr]^{m}\\
\widetilde{\cU}\ar@/_/[u]_{\tilde{\sigma}} \ar[r]^{\mathrm{pr}_1} & \cU & \SS_{4,10}^-\\
                 }
$$

\vskip 3pt

\noindent \emph{Proof of Theorem \ref{thm:kodaira4}.} Since $\widetilde{\cU}$ is rational by Corollary \ref{cor:deg2} and $\tilde{\mathfrak{q}}$ is generically a fibration in rational varieties which possesses a rational section, it follows that $\widetilde{\cR}$ is a rational curve over the function field of $\widetilde{\cU}$ having a rational point. We conclude that $\widetilde{\cR}$ is a rational variety. Using Proposition \ref{prop:dom3} and the diagram above, we have a dominant map $\widetilde{\cR}\dashrightarrow \SS_{4,10}^-$, therefore we obtain that $\SS_{4,10}^-$ is unirational.
\hfill $\Box$

\section{Moduli spaces of pointed spin curves}

We recall basic facts concerning the moduli space $\SS_{g,n}$ of $n$-pointed spin curves of genus $g$, largely following \cite{Cor}, \cite{F2}. If $(X, p_1, \ldots, p_n)$ is an $n$-pointed nodal curve, a smooth
rational component $E\subseteq X$ is said to be \emph{exceptional} if
$\bigl|E\cap \overline{(X\setminus E)}\bigr|=2$ and no marked point $p_i$ lies on $E$. The curve $X$ is said to be
\emph{quasi-stable} if any two exceptional components of $X$ are
disjoint. A quasi-stable curve is thus obtained from a stable pointed curve
by blowing-up each node at most once. 

\begin{definition}\label{def:prymstructures} An $n$-pointed \emph{stable spin curve} of genus
$g$ consists of a triple $(X, p_1, \ldots, p_n, \vartheta, \beta)$, where $(X, p_1, \ldots, p_n)$ is an $n$-pointed  genus
$g$ quasi-stable curve, $\vartheta\in \mathrm{Pic}^{g-1}(X)$ is a line bundle
of total degree $g-1$ such that $\vartheta_{E}=\OO_E(1)$ for every exceptional
component $E\subseteq X$, and $\beta\colon \vartheta^{\otimes 2}\rightarrow
\omega_X$ is a sheaf homomorphism which is generically non-zero along
each non-exceptional component of $X$.
\end{definition}

We denote by $\SS_{g,n}$ the moduli stack of $n$-pointed spin curves of genus $g$, which splits into two connected components $\SS_{g,n}^-$ and $\SS_{g,n}^+$ depending on the parity of $h^0(X, \vartheta)$. One has a finite morphism $\pi\colon \SS_{g,n}\rightarrow \mm_{g,n}$ of degree $2^{2g}$. We also set $\SS_{g, [n]}:=\SS_{g,n}/\mathfrak{S}_n$, where the action is by permuting the marked points. Let $\lambda:=\pi^*(\lambda)$ and $\psi_i:=\pi^*(\psi_i)\in CH^1(\SS_{g,n})$ be the standard codimension one tautological classes; we refer to \cite{AC} for background on divisor classes on $\mm_{g,n}$.  

\vskip 3pt

\subsection{Boundary divisors on $\SS_{g,n}$} We introduce the standard notation for the boundary divisors on both components of $\SS_{g,n}^+$ as follows. For $0\leq i\leq g$ and a subset $S\subseteq [n]$, we denote by $A_{i:S}\subseteq \SS_{g,n}^+$ (respectively $B_{i:S}\subseteq \SS_{g,n}^+$) the closure of the locus consisting of spin curves $[C\cup_y D, p_1, \ldots, p_n, \vartheta_C, \vartheta_D]$, where $C$ and $D$ are smooth curves of genera $i$ and $g-i$ meeting at a point $y$, both theta-characteristics $\vartheta_C$ and $\vartheta_D$ are even (respectively odd), and the marked points lying on the component $C$ are precisely those labelled by the set $S$. We set $\alpha_{i:S}:=[A_{i:S}]$ and $\beta_{i:S}:=[B_{i:S}]$. Note that $B_{0:S}=\emptyset$ for all $S$. For $2\leq s\leq n$, we set 
\begin{equation}\label{eq:alpha_s}
\alpha_{0:s}:=\sum_{S\in {[n] \choose s}} \alpha_{0:S}\in CH^1(\SS_{g,n}^+).
\end{equation}

\vskip 1pt

Similarly, we let $A_{i:S}\subseteq \SS_{g,n}^-$ be the closure of the locus consisting of those spin curves $[C\cup_y D, p_1, \ldots, p_n, \vartheta_C, \vartheta_D]$, where $C$ and $D$ are smooth curves of genera $i$ and $g-i$ meeting at a point $y$, the theta-characteristic $\vartheta_C$ is odd, whereas $\vartheta_D$ is even, and the marked points lying on $C$ are precisely those labelled by $S$. We set $\alpha_{i:S}:=[A_{i:S}]$. To ensure uniformity with the notation from \cite{FV2}, we also set $B_{i:S}:=A_{g-i:S^c}$. Note that in this case, $A_{0:S}=\emptyset$ for all $S$. Then for $s\geq 2$,  we can define the class 
\begin{equation}\label{eq:beta_s}
\beta_{0:s}:=\sum_{S\in {[n] \choose s}} \beta_{0:S}\in CH^1(\SS_{g,n}^-).
\end{equation}

We recall the description of the pull-back of the boundary divisor $\Delta_{\mathrm{irr}}$ of $\mm_{g,n}$ under the finite cover $\pi\colon \SS_{g,n}\rightarrow \mm_{g,n}$. For a point $[X, p_1, \ldots, p_n, \vartheta, \beta]\in \SS_{g,n}$ corresponding to a stable model $\mbox{st}(X)=C_{yq}:=C/y\sim q$,
with $[C, y, q]\in \cM_{g-1, 2}$, there are two possibilities for the spin structure $\vartheta$,
depending on whether $X$ has an exceptional component or not.
If $X$ has no exceptional component and $\vartheta_C:=\nu^*(\vartheta)$, where $\nu\colon C\rightarrow X$
denotes the normalization map, then $\vartheta_C^{\otimes 2}=\omega_C(y+q)$.
For each choice of $\vartheta_C$ as above, there
is precisely one choice of gluing the fibres $\vartheta_C(y)$ and
$\vartheta_C(q)$ such that $h^0(X, \eta)$ has a prescribed parity. We denote by $A_0$ the
closure in $\SS_{g,n}$ of the locus of spin curves as above.

If $X=C\cup_{\{y, q\}} E$, where $E$ is an exceptional component,
then the restriction $\vartheta_C$ is a theta-characteristic on $C$.
Let $B_0\subseteq \SS_{g,n}$ be the closure of the locus of spin curves
$$\bigl[C\cup_{\{y, q\}} E, p_1, \ldots, p_n, \  \ \vartheta_C\in \sqrt{\omega_C}, \ \vartheta_E=\OO_E(1)\bigr]\in \SS_{g,n}.$$ 

\begin{remark} One has an isomorphism $B_0\cong \SS_{g-1:[2]}$  valid for both components of $\SS_g$.
\end{remark} 

If $\alpha_0:=[A_0], \beta_0:=[B_0]\in CH^1(\SS_{g,n})$,  then
$\pi^*\bigl(\delta_{\mathrm{irr}}\bigr)=\alpha_0+2\beta_0$. In particular, $B_0$ is the ramification divisor of the map $\pi$. Coupled with the formula \cite{Log} for $K_{\mm_{g,n}}$, this yields 
\begin{equation}\label{eq:can_class}
K_{\SS_{g,n}}=13\lambda-2\alpha_0-3\beta_0+\sum_{i=1}^n \psi_i-2\sum_{i, S\subseteq [n]} \bigl(\alpha_{i:S}+\beta_{i:S}\bigr)-\alpha_{1:\emptyset}-\beta_{1:\emptyset} \in CH^1(\SS_{g,n}).
\end{equation}

\subsection{Extension of pluricanonical forms} The Kodaira dimension of $\SS_{g,n}$, being an invariant of the coarse moduli space $\SS_{g,n}$ rather than of the stack, is defined by passing to a resolution of singularities $\epsilon\colon \widehat{\mathcal{S}}_{g,n}\rightarrow \SS_{g,n}$. For a $\mathbb Q$-factorial normal projective variety $X$, we denote by $\kappa(X)$ its Kodaira dimension and by $\kappa(X,K_X)$ the Kodaira--Iitaka dimension of its canonical bundle. In order to work directly on the space $\SS_{g,n}$ whose geometry we can control, we need to know that pluricanonical forms on $\SS_{g,n}$ extend to any resolution, therefore $\kappa(\SS_{g,n})=\kappa(\SS_{g,n}, K_{\SS_{g,n}})$. Such a result has been established for $\mm_g$ in \cite{HM}, for $\SS_g$ in \cite{Lu}, and for $\mm_{g,n}$ in \cite{Log}.

\begin{proposition}\label{prop:extension_canform} We fix $g\geq 4$ and $n\geq 0$. Then for any $\ell\geq 0$ one has an isomorphism 
$$\epsilon^*\colon H^0\bigl(\SS_{g,n}, K_{\SS_{g,n}}^{\otimes \ell}\bigr) \stackrel{\cong}\longrightarrow H^0\bigl(\widehat{\cS}_{g,n}, K_{\widehat{\cS}_{g,n}}^{\otimes \ell}\bigr).$$
\end{proposition}
\begin{proof}
We explain the local description of the map $\pi\colon \SS_{g,n}\rightarrow \mm_{g,n}$ around a spin curve $[X, p_1, \ldots, p_n, \vartheta, \beta]$. We denote by $E_1, \ldots, E_r$ the exceptional components of $X$, by $\nu\colon X\rightarrow C$ the map contracting $E_1, \ldots, E_r$ and we set $q_i:=\nu(E_i)\in C_{\mathrm{sing}}$, for $i=1,\ldots, r$. Let $\mathbb C^{3g-3+n}_{\tau}$ be the versal deformation space of $(X,p_1, \ldots, p_n, \vartheta, \beta)$, where the coordinates $(\tau_1, \ldots, \tau_{3g-3+n})$ are chosen in such a way that for $1\leq i\leq r$ the locus $(\tau_i=0)$ corresponds to those deformations which preserve the components $E_i$. Then $\mathbb C^{3g-3+n}_t\cong \mbox{Ext}^1\bigl(\Omega_C(x_1+\cdots+x_n),\OO_C\bigr)$ is the versal defomation space of $[C,\nu(p_1), \ldots, \nu(p_n)]$, where the coordinates $(t_1, \ldots, t_{3g-3+n})$ are chosen in such a way that for $1\leq i\leq r$ the divisor $(t_i=0)\subseteq \mathbb C^{3g-3+n}_t$ corresponds to the deformation failing to smooth the node $q_i$. 

\vskip 3pt

The morphism $\pi\colon \SS_{g,n}\rightarrow \mm_{g,n}$ is locally given by the following map
$$\frac{\mathbb C^{3g-3+n}}{\mbox{Aut}(X,p_1,\ldots, p_n, \vartheta, \beta)}\longrightarrow \frac{\mathbb C^{3g-3+n}}{\mbox{Aut}\bigl(C, \nu(p_1), \ldots, \nu(p_n)\bigr)}, \ \ \ t_i=\tau_i^2, \ i\leq r, \ \ \ t_i=\tau_i, i\geq r+1.$$
The singularities of the quotient $\mathbb C^{3g-3+n}_{\tau}/\mbox{Aut}(X,p_1,\ldots, p_n,\vartheta)$ are studied via the \emph{Reid-Tau criterion} \cite[p.27]{HM}. 
For an automorphism $\phi\in \mbox{Aut}(X,p_1,\ldots,p_n,\vartheta, \beta)$ having order $m$, we define its \emph{age}  as the quantity 
$\mbox{age}(\phi)=\frac{a_1}{m}+\cdots+\frac{a_{3g-3+n}}{m}$,
where $\zeta^{a_1}, \ldots,  \zeta^{a_{3g-3+n}}$ are the eigenvalues of $\phi$, with $\zeta=e^{\frac{2\pi i}{m}}$. An automorphism $\phi$ with $\mbox{age}(\phi)\geq 1$ leads to a canonical singularity. Every automorphism of $\phi\in \mbox{Aut}(X,p_1, \ldots, p_n, \vartheta, \beta)$ induces an automorphism $\phi_{\mathrm{up}}$ of the \emph{unpointed} spin curve $(X, \vartheta, \beta)$. Clearly, one has $\mbox{age}(\phi_{\mathrm{up}})\leq \mbox{age}(\phi)$. Note also that contracting a rational curve under the stabilization map  $\nu\colon X\rightarrow C$ does not change the age of the automorphism.  We now use \cite[Theorem 3.1]{Lu} stating that if $(X, \vartheta, \beta)$ admits an automorphism $\phi$ with $\mbox{age}(\phi)<1$, then necessarily $C$ has an elliptic tail $J_0$ with $j$-invariant  $j(J_0)=0$ and that $\vartheta_{J_0}\cong \OO_{J_0}$ and $\phi$ acts non-trivially on $J_0$. Therefore the same conclusion applies for $\phi\in \mbox{Aut}(X, p_1, \ldots, p_n, \vartheta, \beta)$.

\vskip 3pt

We are left with the exceptional case singled out in \cite{Lu}, that is, when $X$ has an ellipic tail $J_0$ as above. If one of the marked points $p_i$ lies on $J_0$, then the resulting singularity is canonical, cf. \cite[Theorem 2.5]{Log}. If $p_1, \ldots, p_n\in X\setminus J_0$, the argument in \cite[Theorem 4.1]{Lu} or in \cite[p.43-44]{HM}  applies mutatis mutandis to show that the non-canonical singularity of $\SS_{g,n}$ induced by $\phi$ does not impose global adjunction conditions, that is, every pluricanonical form on $\SS_{g,n}$ can be lifted to a resolution.
\end{proof}

\begin{remark}\label{rmk:extgen1} A result analogous to Proposition \ref{prop:extension_canform} has been established in \cite[Theorem 5.61]{Kr} for the space $\SS_{1,n}^+$. Probably the analysis in Proposition \ref{prop:extension_canform} can be extended to cover the case $g=2,3$ as well.
\end{remark}

\subsection{The universal odd theta-characteristic} We now compute the class of an important effective divisor on $\SS_{g,1}^-$, namely the closure $\th1$ of the locus $$\Theta_{g,1}^-:=\Bigl\{[C,p,\vartheta]\in \cS_{g,1}^-: H^0(C, \vartheta(-p))\neq 0\Bigr\}$$
of points in the supports of odd theta-characteristics. The forgetful map $\Theta_{g.1}^-\rightarrow \cS_g^-$ is a generically finite cover branched along the divisor $Z_g$ considered in \cite{FV2} and consisting of odd spin curves $[C, \vartheta]\in \cS_g^-$ such that $\vartheta$ is non-reduced.

\begin{remark}\label{rmk:irr}
The divisor $\Theta_{g,1}^-$ is intimately related to the stratum of abelian differentials 
$$\H_g\bigl(2^{g-1}\bigr):=\Bigl\{[C,p_1, \ldots, p_{g-1}]\in \cM_{g,2g-2}: \omega_C=\OO_C\bigl(2p_1+\cdots+2p_{g-1}\bigl)\Bigr\}.$$
There is a dominant forgetful morphism $\H_g\bigl(2^{g-1}\bigr)\rightarrow \Theta_{g,1}^-$. Since $\H_g\bigl(2^{g-1}\bigr)$ is irreducible, see \cite{KZ}, it follows that $\Theta_{g,1}^-$ is also irreducible for $g\geq 2$.
\end{remark}

\begin{theorem}\label{thm:univ_theta}
The following formula holds for $g\geq 2$:
$$\bigl[\th1]=\frac{\lambda}{4}+\frac{\psi}{2}-\frac{\alpha_0}{16}-\frac{1}{2}\sum_{i=1}^{g-1}\alpha_{i:\emptyset}\in CH^1(\SS_{g,1}^-).$$
\end{theorem}
\begin{proof}
We expand the class of the divisor $\th1$ in the standard basis of $CH^1(\SS_{g,1}^-)$
\begin{equation}\label{eq:expansion}
\bigl[\th1\bigr]=\bar{\lambda} \cdot \lambda+\bar{\psi} \cdot \psi-\bar{\alpha}_0\cdot \alpha_{0}-\bar{\beta}_0\cdot \beta_0-\sum_{i=1}^{g-1}\bigl(\bar{\alpha}_{i:1}\cdot \alpha_{i:1}+\bar{\alpha}_{i:\emptyset}\cdot \alpha_{i:\emptyset}\bigr) \in CH^1(\SS_{g,1}^-),
\end{equation}
and determine the coefficients by intersecting both sides of the equality (\ref{eq:expansion}) with standard test curves lying in the boundary of $\SS_{g,1}^-$.

\vskip 3pt
For $1\leq i\leq g-1$, we fix general spin curves $[C, \vartheta_C^{+}]\in \cS_{i}^+$ and $[D, q, y, \vartheta_D^-]\in \cS_{g-i,2}^-$ and we form the test curve
$$F_i:=\Bigl\{[C\cup_y D, q\in D, \vartheta_C^{+}, \ \vartheta_D^{-}]:y\in C\Bigr\}\subseteq \SS_{g,1}^-.$$
Note that $F_i\cdot \alpha_{g-i:1}=-\mbox{deg}(\omega_C)=2-2i$ and that $F_i$ has intersection number zero with all remaining generators of $CH^1(\SS_{g,1}^-)$. Then we fix general spin curves $[C, \vartheta_C^{-}]\in \cS_{i}^-$ and $[D, q, y, \vartheta_D^{+}]\in \cS_{g-i,2}^+$ and consider the following test curve 
$$G_i:=\Bigl\{[C\cup_y D, q\in D, \vartheta_C^{-}, \ \vartheta_D^{+}]:y\in C\Bigr\}\subseteq \SS_{g,1}^-.$$
Then $G_i\cdot \alpha_{i:\emptyset}=2-2i$ and the other intersections with the generators of $CH^1(\SS_{g,1}^-)$ equal zero.

\vskip 3pt

We claim that $F_i$ is disjoint from $\th1$, in particular $\bar{\alpha}_{g-i:1}=0$, for $i=2, \ldots, g-1$. Indeed, assume  there exists a point $t\in \th1\cap F_i$ corresponding to a point $y\in C$. Then there exists a pair of non-zero sections $(\sigma_C, \sigma_D)$ corresponding to a limit theta-characteristic on $C\cup_y D$ such that $\sigma_D(q)=0$. We may assume that $h^0(D, \vartheta_D^-)=1$ and that neither $y$ nor $q$ are zeroes of the unique section in $H^0(D, \vartheta_D^-)$. One has 
$\sigma_C\in H^0\bigl(C, \vartheta_C^+((g-i)\cdot y)\bigr)$, $\sigma_D\in H^0\bigl(D, \vartheta_D^-(i\cdot y)\bigr)$ and, by the definition of a limit linear series, that $\mbox{ord}_y(\sigma_C)+\mbox{ord}_y(\sigma_D)\geq g-1$. Since $\sigma_D(q)=0$, it follows that $\mbox{ord}_y(\sigma_D)\leq i-1$, therefore $\mbox{ord}_y(\sigma_C)\geq g-i$, from which it follows that $H^0(C, \vartheta_C^+)\neq 0$, which contradicts the generality of $[C, \vartheta_C^+]\in \cS_{i,1}^+$ . This proves that $F_i\cdot \th1=0$.

\vskip 4pt

Next we claim that $G_i\cdot \th1=i-1$, therefore $\bar{\alpha}_{i:\emptyset}=\frac{1}{2}$, for $i=2, \ldots, g-1$. Let $t\in \th1\cap G_i$ be a point corresponding to a point $y\in C$. Then we have  $\sigma_C\in H^0\bigl(C, \vartheta_C^-((g-i)\cdot y\bigr))$ and $\sigma_D\in H^0\bigl(D,\vartheta_D^+(i\cdot y)\bigr)$ and, again, $\mbox{ord}_y(\sigma_C)+\mbox{ord}_y(\sigma_D)\geq g-1$. Since, by generality,  we may assume  $H^0\bigl(D, \vartheta_D(y-q)\bigr)=0$, it follows that $\mbox{ord}_y(\sigma_D)\leq i-2$, therefore $\mbox{ord}_y(\sigma_C)\geq g-i+1$, that is, $y\in \mbox{supp}(\eta_C^-)$. This yields $i-1$ points for the intersection $G_i\cdot \th1$, each corresponding to the points in the support of $\vartheta_C^-$. To see that each of them counts with multiplicity one, we reason along the lines of \cite[Proposition 5.1]{FV2}.

\vskip 4pt

Now we argue that $\th1$ is disjoint from two elliptic pencils in the boundary of $\SS_{g,1}^-$. Let $f\colon \mbox{Bl}_9(\PP^2)\rightarrow \PP^1$ be the family of elliptic curves induced by blowing-up a pencil of plane cubics at its $9$ base points and denote by $\tau\colon \PP^1\rightarrow \mbox{Bl}_9(\PP^2)$ a section corresponding to one of the base points of the pencil. We fix a general spin curve $[C, y, q, \vartheta_C^{+}]\in \cS_{g-1,2}^+$ and set
$$F_0:=\Bigl\{\bigl[C\cup_{y\sim \tau(t)} f^{-1}(t),\
 \ \vartheta_C^{+},\ \
\vartheta_{f^{-1}(t)}=\OO_{f^{-1}(t)}\bigr]: t\in
\PP^1\Bigr\}\subseteq \SS_{g,1}^-.$$ We find $F_0\cdot \alpha_{1:\emptyset}=-1$ and
$F_0\cdot \lambda=1$. For each of
the $12$ points $t_{\infty}\in \PP^1$ corresponding to
singular fibres of $f$, the spin structure is locally free on $C \cup
f^{-1}(t_{\infty})$, therefore we obtain a point lying in the divisor $A_0$. We therefore conclude that $F_0\cdot \beta_0=0$ and, accordingly, $F_0\cdot \alpha_0=12$.

\vskip 3pt

A second elliptic pencil is obtained by choosing a general element $[C, y, q, \vartheta_C^-]\in \cS_{g-1,2}^-$. On $f^{-1}(t)$ one takes each of the three even theta-characteristics. This induces the family
$$G_0:=\Bigl\{\bigl[C\cup_{y\sim \tau(t)} f^{-1}(t), \ \vartheta_C^{-},
\  , \vartheta_{f^-1(t)}^{+}\in \gamma^{-1}
[f^{-1}(t)]\bigr]: t\in \PP^1\Bigr\}\subseteq \SS_{1,1}^-,$$ with $\gamma\colon \SS_{1, 1}^+\rightarrow \mm_{1, 1}$ being the degree $3$ map forgetting the spin structure. We obtain $G_0\cdot \lambda=3$ and
$G_0\cdot \beta_{1:\emptyset}=-3$.  The map $\gamma$ being simply ramified over the point
corresponding to $j$-invariant $\infty$, we conclude that $G_0\cdot
\alpha_0=12$ and therefore  $G_0\cdot \beta_0=12$.

\vskip 3pt

It is a simple exercise in limit linear series to show that $\th1$ is disjoint from both $F_0$ and $G_0$. Accordingly, we obtain the following relations 
\begin{equation}\label{eq:ellipticpencil}
\bar{\lambda}-12\bar{\alpha}_0+\bar{\alpha}_{1:\emptyset}=0 \ \ \mbox{ and } \ \ \ \bar{\lambda}-4\bar{\alpha}_0-4\bar{\beta}_0+\bar{\beta}_{1:\emptyset}=0.
\end{equation}

\vskip 6pt  

Next we fix a general pointed even spin curve $[C,y,\vartheta_{C}^+]\in \cS_{g-1,1}^+$ and consider the map $\nu\colon \mm_{1,2}\rightarrow \SS_{g,1}^-$ given by $[J, y,q]\mapsto \bigl[J\cup_{y} C, q, \vartheta_J^{-}=\OO_J, \ \vartheta_C^+\bigr]\in \SS_{g,1}^-$, where $J$ denotes an elliptic curve. The following relations are immediate:
$$\nu^*(\lambda)=\lambda, \ \nu^*(\psi)=\psi_q, \nu^*(\alpha_0)=\delta_{\mathrm{irr}}, \ \nu^*(\beta_0)=0, \ \nu^*(\alpha_{1:1})=-\psi_x, \ \ \nu^*(\alpha_{1:\emptyset})=\delta_{0:yq}.$$ Again, it is an easy exercise to show that $\nu^*(\th1)=0$.
Inside $CH^1(\mm_{1,2})\cong \mathbb Q\langle \psi_y, \delta_{0:yq}\rangle$ one has the following well-known relations \cite[Proposition 1.9]{AC}:
$$\psi_y=\psi_q, \ \lambda=\psi_y-\delta_{0:yq}, \ \ \delta_{\mathrm{irr}}=12(\psi_y-\delta_{0: yq}).$$
We obtain that $0=\bigl[\nu^*(\th1)\bigr]=\bigl(\bar{\lambda}+\bar{\psi}-12\bar{\alpha}_0\bigr)\cdot \psi_y-\bigl(\bar{\lambda}-12\bar{\alpha}_0+\bar{\alpha}_{1:\emptyset}\bigr)\cdot\delta_{0:yq}\in CH^1(\mm_{1,2})$. It is a consequence of \cite[Proposition 4.3]{F3} that $\bar{\alpha}=\frac{1}{4}$, while obviously   $\bar{\psi}=\frac{1}{2}$. It follows that $\bar{\alpha}_0=\frac{1}{12}\bigl(\frac{1}{2}+\frac{1}{4}\bigr)=\frac{1}{16}$. It then follows from the first relation in (\ref{eq:ellipticpencil}) that $\bar{\alpha}_{1:\emptyset}=\frac{1}{2}$. 

\vskip 4pt

We are now in a position to determine the coefficient $\bar{\beta}_0$ and to that end, we use that in \cite[Theorem 0.3]{F3} the class of the pushforward $\pi_*\bigl([\th1]\bigr)\in CH^1(\mm_{g,1})$ has been  determined. In particular, its $\delta_{\mathrm{irr}}$-coefficient is equal to $-2^{2g-6}$. From the set-theoretic description of both divisors $\alpha_0$ and $\beta_0$, clearly
$\pi_*(\alpha_0)=2^{2g-2}\delta_{\mathrm{irr}}$ and $\pi_*(\beta_0)=2^{g-2}(2^{g-1}-1)\delta_{\mathrm{irr}}$, therefore 
$2^{2g-2}\bar{\alpha}_0+2^{g-2}(2^{g-1}-1)\bar{\beta}=2^{2g-6}$. Since $\bar{\alpha}_0=\frac{1}{16}$, we obtain $\bar{\beta}_0=0$. Finally, from the second equation in (\ref{eq:ellipticpencil}), we obtain $\beta_{1:\emptyset}=0$. This completes the proof. 
\end{proof}  

\subsection{The uniruledness of $\th1$ for small $g$} In the range when the general curve of genus $g$ lies on a $K3$ surface, we can show that the divisor $\th1$ is uniruled. In fact, we have a more precise result which will be used when determining the Kodaira dimension of $\SS_{g,1}^-$:

\begin{theorem}\label{thm:uniruledK3}
For $g\leq 9$ or $g=11$, the universal theta-characteristic $\th1$ is uniruled. Precisely through a general point of $\th1$ there passes a rational curve $\Gamma\subseteq \th1$, such that $\Gamma\cdot \th1=0$.
\end{theorem} 

\begin{proof}
We fix a general point $[C, p, \vartheta]\in \cS_{g,1}^-$ and we may assume $h^0(C,\vartheta)=1$ and that $\mbox{supp}(\vartheta)$ consists of $g-1$ distinct points $p=p_1$, $p_2, \ldots, p_{g-1}$. When $g\leq 11$ and $g\neq 10$,  there exists a smooth $K3$ surface $S\supseteq C$. We may even assume $\mbox{Pic}(S)\cong \mathbb Z\cdot C$.  Note that in the embedding $S\stackrel{|C|}\longrightarrow \PP^g$, the points $p_1, \ldots, p_{g-1}$ span a codimension two subspace $\Pi=\langle p_1, \ldots, p_{g-1}\rangle \subseteq \PP^g$ such that $\Pi\cdot S=2(p_1+\cdots+p_{g-1})$.  Following \cite[Theorem 3.10]{FV2}, we consider the pencil of hyperplanes $\{H_t\}_{t\in\PP^1}$ in $\PP^g$ containing $\Pi$. This induces a rational curve in moduli 
\begin{equation}\label{eq:pencil_odd}
\Gamma:=\Bigl\{\bigl[C_t:=H_t\cap S, \ p_1, \ \ \vartheta_t=\OO_{C_t}(p_1+\cdots+p_{g-1})\bigr]: t\in \PP^1\Bigr\}\subseteq \th1.
\end{equation}
It follows that $\Gamma\cdot \lambda=g+1$, $\Gamma\cdot \beta_0=g-1$ and that $\Gamma \cdot \alpha_0=4g+20$, see \cite[Corollary 3.8]{FV2}. The points in the intersection $\Gamma\cdot \beta_0$ correspond to those hyperplanes $H_t$ such that the intersection $H_t\cap S$ is nodal at one of the points $p_i$. All the elements in the pencil $\{H_t\cap S\}_{t\in \PP^1}$ are irreducible and have at most one singular point.

\vskip 3pt

Observe that $\Pi\cap \mathbf{T}_{p}(S)$ is a $1$-dimensional space and this tangent direction may be regarded as a point $v$ on the exceptional divisor $E_p$ of the blow-up $\widetilde{S}=\mathrm{Bl}_{\{p_1,\ldots, p_{g-1}\}}(S)$. Set $\epsilon\colon S':=\mbox{Bl}_{v}(S')\rightarrow S$ and denote by $E_p'\subseteq S'$ the proper transform of $E_p$. The family of pointed spin curves in (\ref{eq:pencil_odd}) is then induced by the fibration $f'\colon S'\rightarrow \PP^1$, where $$f'^*\bigl(\OO_{\PP^1}(1)\bigr)\cong \OO_{S'}(\epsilon^*(C)-E_p'-E_{p_2}-\cdots-E_{p_{g-1}}).$$  We have $\Gamma\cdot \psi=-\bigl(E_p'\bigr)^2=2$. Using Theorem \ref{thm:univ_theta}, we obtain 
$\Gamma\cdot \th1=\frac{g+1}{4}+1-\frac{4g+20}{16}=0$.
\end{proof}

\subsection{The positivity of $K_{\SS_{g,n}}$} In order to prove that $K_{\SS_{g,n}}^{\pm}$ is big, we introduce some known effective divisors on moduli spaces related to $\SS_{g, n}^{\pm}$. We first recall the class \cite{F1} of the compactification $\overline{\Theta}_{\mathrm{null}}$ of the theta-null divisor $\Theta_{\mathrm{null}}:=\Bigl\{[C, \vartheta]\in \cS_g^+: H^0(C, \vartheta)\neq 0\Bigr\}$:
\begin{equation}\label{eq:theta_null}
[\overline{\Theta}_{\mathrm{null}}]=\frac{1}{4}\lambda-\frac{1}{16}\alpha_0-\frac{1}{2}\sum_{i=1}^{\lfloor \frac{g}{2}\rfloor} \beta_i\in CH^1(\SS_g^+).
\end{equation}
On $\SS_g^-$, we have the class of the closure of the locus $Z_g$ of odd spin curves with a non-reduced support \cite[Theorem 0.4]{FV2}:
\begin{equation}\label{eq:nonred_theta}
[\overline{Z}_g]=(g+8)\lambda-\frac{g+2}{4}\alpha_0-2\beta_0-\sum_{i=1}^{g-1} 2(g-i)\alpha_i\in CH^1(\SS_g^-).
\end{equation}
Observe that $Z_g$ is the branch divisor of the generically finite forgetful map $\Theta_{g,1}^-\rightarrow \cS_g^-$.  On $\mm_{g,n}$, we have the class of Logan's divisor \cite[Theorems 5.3-5.7]{Log} defined as the closure in $\mm_{g,g}$ of the following locus  $D_g:=\Bigl\{[C,p_1, \ldots, p_g]\in \cM_{g.g}: h^0\bigl(C, \OO_C(p_1+\cdots+p_g)\bigr)\geq 2\Bigr\}$:
$$
[\overline{D}_g]=-\lambda+\sum_{i=1}^g\psi_i-\sum_{i\geq 0, S}{\bigl||S|-i\bigr|+1\choose 2} \delta_{i:S}\in CH^1(\mm_{g,g}).
$$

We are now in a position to prove Theorem \ref{thm:kodaira_pointed}:

\vskip 5pt

\noindent \emph{Proof of Theorem \ref{thm:kodaira_pointed} for $\SS_{g,n}^+$.}
We first treat the case $n\geq g$. We shall use the fact that the class $\sum_{i=1}^n \psi_i$ is big on $\SS_{g,n}^+$. For a subset $S\subseteq [n]$ with $|S|=g$, let $\pi_S\colon \mm_{g,n}\rightarrow \mm_{g,g}$ be the morphism forgetting all the marked points in $S^c$ and we consider the symmetrized version of the pull-back of Logan's effective divisor
$D_{g,n}:=\frac{1}{{n\choose g}} \pi_S^*(D_g)$ on $\mm_{g,n}$. Then using standard formulas for the pull-backs of boundary classes under forgetful maps \cite{AC}, we compute
\begin{equation}\label{eq:symm_logan}
\bigl[\overline{D}_{g,n}\bigr]=-\lambda+\frac{g}{n}\sum_{i=1}^n \psi_i-\frac{g(g-3+2n)}{n(n-1)}\sum_{|S|=2}\delta_{0:S}-\cdots \in CH^1(\mm_{g,n}).
\end{equation}
We choose an effective divisor $D$ on $\mm_g$ with $[D]=s\lambda-\delta_{\mathrm{irr}}-\sum_{i=1}^{\lfloor \frac{g}{2}\rfloor} b_i\cdot \delta_i\in CH^1(\mm_g)$, with $b_i\geq 1$, and where $s=6+\frac{12}{g+1}$  when $g+1$ is composite (in which case $D$ is a multiple of a Brill-Noether divisor \cite[Theorem 1]{EH}), respectively $s=\frac{17}{2}$ when $g=4$, respectively $s=\frac{47}{6}$ when $g=6$ (in both these cases $D$ being a multiple of the corresponding Petri divisor, see \cite[Theorem 2]{EH}) and finally $s=7$ when $g=10$, in which case $D$ is the $K3$ divisor considered in \cite{FP}. Denoting by $\sigma\colon \SS_{g,n}\rightarrow \mm_g$ respectively by $\mu\colon \SS_{g,n}\rightarrow \SS_g$ the natural forgetful maps and recalling that $\pi\colon \SS_{g,n}^+\rightarrow \mm_{g,n}$ is the covering map, we form the following effective combination on $\SS_{g,n}^+$
\begin{align*}
8\bigl[\mu^*(\overline{\Theta}_{\mathrm{null}})\bigr]+\frac{2n(n-1)}{g(g-3+2n)}\bigl[\pi^*(\overline{D}_{g,n})\bigr]+\frac{3}{2}\bigl[\sigma^*(D)\bigr]=\Bigl(\frac{3s+4}{2}-
\frac{2n(n-1)}{g(g-3+2n)}\Bigr)\lambda+\\
\frac{2(n-1)}{g-3+2n}\sum_{i=1}^n \psi_i-2\alpha_0-3\beta_0-2\alpha_{0:2}-\sum_{i,S} \bigl(\overline{\alpha}_{i:S}\cdot \alpha_{i:S}+\overline{\beta}_{i:S}\cdot \beta_{i:S}\bigr) \in CH^1(\SS_{g,n}^+),
\end{align*}
where $\overline{\alpha}_{i:S}, \overline{\beta}_{i:S}\geq 2$ in all cases except when $i=0$ and $|S|=2$, in which case the corresponding coefficients are equal to zero.  We compare this formula to the expression (\ref{eq:can_class}). We observe that the coefficient of $\sum_{i=1}^n \psi_i$ is smaller than one, whereas the range of $(g,n)$ in the statement of Theorem \ref{thm:kodaira_pointed} is choses such that the $\lambda$-coefficient is less than $13$. It follows that $K_{\SS_{g,n}^+}$ is big. Coupled with Proposition \ref{prop:extension_canform}, this shows that $\SS_{g,n}^+$ is of general type in this range.

\vskip 4pt

For $g=8$, it is easy to see that $K_{\SS_{8,1}^+}$ can be expressed as a combination with positive coefficients of $[\sigma^*(D)]$, $\bigl[\mu^*(\overline{\Theta}_{\mathrm{null}})\bigr]$, $\psi$, the pull-back of the Weierstrass divisor under the forgetful map $\SS_{8,1}^+\rightarrow \mm_{8,1}$ and certain boundary divisors. It follows that $K_{\SS_{8,1}^+}$ is big.

\vskip 4pt

Finally, on $\SS_{7,3}^+$, we consider the closure of the divisor $\mathfrak{D}_{2:2:3}$ consisting of those pointed spin curves $\bigl[C, p_1, p_2, p_3, \vartheta\bigr]\in \cS_{7,3}^+$, such that there exists a permutation $\tau\in \mathfrak{S}_3$ with $$h^0\bigl(C, \OO_C(2p_{\tau(1)}+2p_{\tau(2)}+3p_{\tau(3)})\bigr)\geq 2.$$ 
Using \cite[Theorems 5.4-5.7]{Log},  its class is equal up to a positive rational constant, to
$$[\overline{\mathfrak{D}}_{2:2:3}]=-3\lambda+12\bigl(\psi_1+\psi_2+\psi_3)-40\cdot \alpha_{0:2}-0\cdot \alpha_0-0\cdot \beta_0-\cdots \in CH^1(\SS_{7,3}^+).$$ By direct calculation, we observe that 
$K_{\SS_{7,3}^+}-8\bigl[\mu^*(\overline{\Theta}_{\mathrm{null}})\bigr]-\frac{1}{12}[\overline{\mathfrak{D}}_{2:2:3}]-\frac{3}{2}[\sigma^*(D)]$ is an effective combination of boundary classes, which completes the proof. \hfill $\Box$

\vskip 6pt

We now move on to prove Theorem \ref{thm:kodaira_pointed} for the moduli space of odd spin curves.

\vskip 5pt

\noindent \emph{Proof of Theorem \ref{thm:kodaira_pointed} for $\SS_{g,n}^-$.} We first treat the case when $g\geq 4$ and $n\geq g$. Since the proof bears similarities with the one for $\SS_{g,n}^+$, we skip a few details. We recall that we considered the Logan divisor $D_{g,n}$ on $\mm_{g,n}$ and the effective divisor $D$ on $\mm_g$ having minimal slope $s$. Denoting by $\pi_i\colon \SS_{g,n}^-\rightarrow \SS_{g,1}^-$ the morphism forgetting all marked points except the one labelled by the index $i$, we set $\overline{\Theta}_{g,n}:=\pi_1^*\bigl(\overline{\Theta}_{g,1}\bigr)+\cdots+\pi_n^*\bigl(\overline{\Theta}_{g,1}\bigr)$. From Theorem \ref{thm:univ_theta}, 

\begin{equation}\label{eq:theta_n}
\bigl[\overline{\Theta}_{g,n}^-\bigr]=\frac{n}{4}\lambda+\frac{1}{2}\sum_{i=1}^n \psi_i-\frac{n}{16}\alpha_0-\beta_{0:2}-\cdots \in CH^1(\SS_{g,n}^-).
\end{equation}

We form the following effective linear combination on $\SS_{g,n}^-$:
\begin{align*}
\frac{8}{n}\bigl[\overline{\Theta}_{g,n}\bigr]+\frac{2(n-4)(n-1)}{g(g-3+2n)}\bigl[\pi^*(\overline{D}_{g,n})\bigr]+\frac{3}{2}\bigl[\sigma^*(D)\bigr]=\Bigl(\frac{3s+4}{2}
-\frac{2(n-4)(n-1)}{g(g-3+2n)}\Bigr)\lambda\\
+\frac{2(n^2+2g-n-2)}{n(g-3+2n)}\sum_{i=1}^n \psi_i-2\alpha_0-3\beta_0-2\alpha_{0:2}-\sum_{i,S}\overline{\alpha}_{i:S}\cdot \alpha_{i:S},
\end{align*}
where $\overline{\alpha}_{i:S}\geq 2$, except when $i=0$ and $|S|=2$, when the corresponding coefficient equals zero. Again, observe that the coefficient of $\sum_{i=1}^n \psi_i$ is smaller than one. Then for $4\leq g\leq 7$, the $\lambda$-coefficient of this effective divisor is smaller than $\frac{13}{2}$ precisely in the range in the statement of Theorem \ref{thm:kodaira_pointed}.   

\vskip 4pt

Assume now $n\leq g$. We use the effective divisors provided by (\ref{eq:nonred_theta}), (\ref{eq:symm_logan}) and (\ref{eq:theta_n}). We can easily show that for every $g\in \{8, 9, 10, 11\}$ and  $n\geq h(g)$, we have that 
$$K_{\SS_{g,n}^-}\in \mathbb Q_{\geq 0}\Bigl\langle \bigl[\overline{\Theta}_{g,n}^-\bigr], \  \bigr[\mu^*(\overline{Z}_g)\bigr], \ \bigl[\sigma^*(D)\bigr], \  \bigl[\pi^*(\overline{D}_{g,n})\bigr], \  \frac{1}{n}\sum_{i=1}^n \psi_i\Bigr\rangle,$$
where the coefficient of $\psi_1+\cdots+\psi_n$ in this expression is positive. It follows that $K_{\SS_{g,n}^-}$ is big. For the boundary cases $\SS_{8,4}^-$ and $\SS_{9,3}^-$, we have the following expressions for the respective canonical divisors:
$K_{\SS_{8,4}^-} -2\bigl[\overline{\Theta}_{8:4}^-\bigr]-\frac{3}{2}[\sigma^*(D)]\in \mathbb Q_{\geq 0}\bigl\langle \alpha_{i:S}\bigr\rangle$, respectively
$$
K_{\SS_{9,3}^-}-2\bigl[\overline{\Theta}_{9:3}^-\bigr]-\frac{10}{7}[\sigma^*(D)]-\frac{1}{14}\bigl[\mu^*(\overline{Z}_{9})\bigr] \in \mathbb Q_{\geq 0}\Bigl\langle \alpha_{i:S}:i\geq 0, S\subseteq [n]\Bigr\rangle.$$ This finishes the proof.
\hfill $\Box$

\vskip 4pt

We now discuss the case $g=11$.

\vskip 4pt

\noindent {\emph{Proof of Theorem \ref{thm:gen11}.} The same consideration as in the proof of Theorem \ref{thm:kodaira_pointed} show that one can represent the canonical class of $\SS_{11,1}^-$ as follows
\begin{equation}\label{eq:can11}
K_{\SS_{11,1}^-}=2\bigl[\overline{\Theta}_{11:1}^-\bigr]+\frac{4}{3}\bigl[\sigma^*(D)\bigr]+\frac{1}{6}\bigl[\mu^*(\overline{Z}_{11})\bigr]+\sum_{i=1}^{10}
\bigl(a_{i}\cdot \alpha_{i}+b_{i}\cdot \alpha_{i:\emptyset}\bigr),
\end{equation}
where $a_{i},b_i\geq 0$. In fact, using \cite[Proposition 10.5]{FP}, there is a $19$-dimensional family of effective divisors on $\mm_{11}$ having slope $7$, therefore we obtain that $\kappa(\SS_{11,1}^-\bigr)\geq 19$. 

It remains to show that $\SS_{11,1}^-$ is not of general type. To that end, we consider the family of rational curves constructed in the proof of Theorem \ref{thm:uniruledK3}. A general member $\Gamma\subseteq \overline{\Theta}_{11,1}^-$ of this family passes through a general point of $\overline{\Theta}_{11,1}^-$ and has the intersection numbers $\Gamma\cdot \lambda=12$, $\Gamma\cdot \psi=2$, $\Gamma\cdot \alpha_0=64$ and $\Gamma\cdot \beta_0=10$, while $\Gamma\cdot \alpha_{i:1}=\Gamma\cdot \alpha_{i:\emptyset}=0$, for $i=1, \ldots, 10$.

 We observe via Theorem \ref{thm:uniruledK3} and (\ref{eq:nonred_theta}) that $\Gamma$ has intersection number zero with \emph{every} divisor appearing on the right hand side of (\ref{eq:can11}), in particular, also $\Gamma\cdot K_{\SS_{11,1}^-}=0$. Since $\Gamma$ fills-up a divisor in $\SS_{11,1}^-$ this rules out the possibility of expressing $K_{\SS_{11,1}^-}$ as an effective combination of an ample and an effective class. In particular, $K_{\SS_{11,1}^-}$ is not big, that is, $\SS_{11,1}^-$ is not of general type. The case of $\SS_{11:[2]}^-$ is similar and we skip the details.
\hfill $\Box$

\subsection{Holomorphic forms on $\SS_{1,n}^+$} We now discuss the geometry of $\SS_{1,n}^+$ and establish Theorem \ref{thm:genus1}. Note that one has an identification $\SS_{1,1}^+\cong X_1(2)\cong \PP^1$.

\vskip 4pt
  
\noindent \emph{Proof of Theorem \ref{thm:genus1}.} One can use (\ref{eq:can_class}) in genus one, together with the relations $12\lambda=\delta_{\mathrm{irr}}$ and $12\psi_i=\delta_{\mathrm{irr}}+12\sum_{i\in S}\delta_{0:S}$, to obtain the following formula \cite[Proposition 3]{BF}:
$$K_{\mm_{1,n}}=(n-11)\lambda+\sum_{a=2}^{n-1}\sum_{S\subseteq [n], |S|=a}(a-2)\delta_{0:S}+(n-3)\delta_{0:n}\in CH^1(\mm_{1,n}).$$
The finite cover $\pi\colon \SS_{1,n}^+\rightarrow \mm_{1,n}$ is branched along the divisor $B_0$. Since we have the isomorphism $\SS_{1,1}^+\cong \PP^1$ at the level of coarse moduli spaces and both divisors $A_0$ and $B_0$ corresponds to points on $\SS_{1,1}^+$, it follows that $\alpha_0=\beta_0\in CH^1(\SS_{1,n}^+)$.  Since $\pi^*(\delta_{\mathrm{irr}})=\alpha_0+2\beta_0$ and $\pi^*(\Delta_{0:S})=A_{0:S}$ for every subset $S\subseteq [n]$ with $|S|\geq 2$, by the Hurwitz formula we can write 
\begin{equation}\label{eq:can_g1}
K_{\SS_{1,n}^+}=(n-11)\lambda+\sum_{a=2}^{n-1}(a-2)\alpha_{0:a}+(n-3)\delta_{0:n}+\beta_0
=\frac{n-7}{4}\alpha_0+\sum_{a=2}^{n}(a-2)\alpha_{0:a}-\alpha_{0:n}.
\end{equation}
In particular, for $n\geq 8$, via Remark \ref{rmk:extgen1}  using the forgetful morphism $\SS_{1,n}^+\rightarrow \SS_{1,1}^+$, we obtain that $\kappa(\SS_{1,n}^+)=\kappa(\SS_{1,n}^+, K_{\SS_{1,n}^+}) \geq \kappa(\SS_{1,1}^+, \alpha_0)=\kappa(\PP^1, \OO_{\PP^1}(1))=1$. On the other hand, as  explained in \cite[Proposition 4]{BF}, since the general fibre of the morphism $\SS_{1,n}^+\rightarrow \SS_{1,1}^+$ has Kodaira dimension zero, by the \emph{easy addition} theorem we write $\kappa(\SS_{1,n}^+)\leq \mbox{dim}(\SS_{1,1}^+)=1$, thus $\kappa(\SS_{1,n}^+)=1$, for $n\geq 8$.

\vskip 4pt

For $n=7$, we obtain from (\ref{eq:can_g1}) that $H^0(\SS_{1,7}^+, K_{\SS_{1,7}}^+)\neq 0$. Finally, $K_{\SS_{1,7}^+}$ admits an effective representative consisting of boundary divisors $A_{0:S}$, where $S\subseteq [7]$. For each such divisor, by varying the $j$-invariant of the elliptic curve corresponding to a general point, one constructs a rational curve $\Gamma_S\subseteq A_{0:S}$ passing through a general point and satisfying $\Gamma_S\cdot \alpha_{0:S}=-1$ and  $\Gamma_S\cdot \alpha_{0:S'}=0$, for every $S'\subseteq [7]$ with $S'\neq S$. It follows that $\kappa(\SS_{1,7}^+)=0$.
\hfill $\Box$
\section{Correspondences between moduli spaces of pointed spin curves.} We now explain a structure result for the spin moduli space $\SS_7^-$ which makes no reference to Mukai's structure theorem \cite{M3} for curves of genus $7$ in terms of the spinorial variety $OG(5,10)\subseteq \PP^{15}$, that is, of the birational isomorphism
$$\mm_7\stackrel{\cong}\dashrightarrow {\bf{G}}(9,15)\dblq \mbox{Aut } OG(5,10).$$  This construction establishes a correspondence between $\SS_7^-$ and $\SS_{3,6}^+$.

%We believe this construction to be of independent interest and to prove to be useful for other applications, e.g. when computing the Chow ring of $\SS_7^-$,  in the spirit of \cite{CL}. 

\vskip 3pt

We denote by $\hs_7^-$ the \emph{Hurwitz spin} stack parametrizing triples $[C, \vartheta, A]$, consisting of an odd spin curve $[C,\vartheta]\in \cS_7^-$ and a pencil $A\in W^1_5(C)$. The general fibre of the forgetful map $\hs_7^-\rightarrow \SS_7^-$ is a smooth curve. By Serre duality $L:=\omega_C\otimes A^{\vee}\in W^2_7(C)$. We consider the induced plane  model 
$$\phi\colon C\stackrel{|L|}\longrightarrow \Gamma\subseteq \PP^2=\PP H^0(C,L)^{\vee}.$$
For a general choice of $[C, \vartheta, A]$, it is straightforward that $\Gamma$ is a nodal septic curve with $8$ nodes, which we denote by $z_1, \ldots, z_8\in \PP^2$. Setting $\vartheta:=\OO_C(p_1+\cdots+p_6)$, it then follows that there exists a smooth plane quadric curve $D\in \bigl|\OO_{\PP^2}(4)\bigr|$ such that 
\begin{equation}\label{eq:int7}
D\cdot \Gamma=2(z_1+\cdots+z_8+p_1+\cdots+p_6).
\end{equation}
Then $\eta:=\OO_D(p_1+\cdots+p_6+z_1+\cdots+z_8)\otimes \omega_D^{-3}\in \mbox{Pic}^2(D)$ is  an even theta-characteristic. 

\vskip 3pt

To obtain a parametrization of $\hs_7^-$ (and ultimately of $\SS_7^-$), we consider the following locally trivial $\PP^5$-bundle 
over $\cS_{3,6}^+$. Let $\mathfrak{u}\colon \C\rightarrow \cS_{3,6}^+$ be the universal curve and $\L\in \mbox{Pic}(\C)$ be the universal spin bundle. We denote by $\sigma_i\colon \cS_{3,6}^+\rightarrow \C$ the section corresponding to the $i$-th marked point and set $\Sigma_i:=\mbox{Im}(\sigma_i)$, for $i=1, \ldots, 6$. Let
$$\P_{3,6}:=\PP\bigl(\mathfrak{u}_*(\L\otimes \omega_{\mathfrak{u}}^3(-\Sigma_1-\cdots-\Sigma_6)\bigr)\longrightarrow \cS_{3,6}^+.$$ 
A point of $\P_{3,6}$ corresponds to an element $\bigl(D, \eta, \overline{p}, z_1+\cdots+z_8\bigr)$, where $[D, \eta]\in \cS_{3,6}^+$ is an even spin curve, $\overline{p}=(p_1,\ldots, p_6)$ and $z_1+\cdots+z_8$ is an effective divisor on $D$ such that 
$$p_1+\cdots+p_6+z_1+\cdots+z_8\in \bigl|\omega_D^3\otimes \eta\bigr|.$$  

%We then consider the variety $\mathcal{Y} \rightarrow \P_{3,6}$ parametrizing  curves $\Gamma\in \bigl|\OO_{\PP^2}(7)(-2z_1-\cdots-2z_8)\bigr|$ such that the relation (\ref{eq:int7}) holds. Note that birationally, $\mathcal{Y}$ is a locally trivial $\PP^2$-bundle over $\P_{3,6}$.

Note that $2(p_1+\cdots+p_6)+2(z_1+\cdots+z_8)\in \bigl|\omega_D^7\bigr|=\bigl|\OO_D(7)\bigr|$. From the exact sequence 
$$0\longrightarrow H^0(\PP^2, \OO_{\PP^2}(3))\stackrel{\cdot D}\longrightarrow H^0(\PP^2, \OO_{\PP^2}(7))\longrightarrow H^0(D, \OO_D(7))\longrightarrow 0,$$
we conclude there exists a plane curve $\Gamma\in \bigl|\OO_{\PP^2}(7)\bigr|$, with $D\cdot \Gamma=2(p_1+\cdots+p_6+z_1+\cdots+z_8)$.
We set $\I_{2\overline{p}+2\overline{z}}:=\I_{p_1}^2\cap \ldots \cap \I_{p_6}^2\cap \I_{z_1}^2\cap \ldots \cap \I_{z_8}^2\subseteq \OO_{\PP^2}$. From the exact sequence 
\begin{equation}\label{eq:10}
0\longrightarrow H^0(\PP^2, \OO_{\PP^2}(3))\longrightarrow H^0\bigl(\PP^2, \I_{2\overline{p}+2\overline{z}}(7)\bigr)\longrightarrow H^0(D,\OO_D)\longrightarrow 0,
\end{equation}
it follows one has a $10$-dimensional linear system of plane curves of degree $7$ cutting out the divisor $2\overline{p}+2\overline{z}$ on $D$.
\begin{definition}
We denote by $\mathcal{Y} \rightarrow \P_{3,6}$ the parameter space of objects $\bigl(D, \eta, \overline{p}, z_1+\cdots+z_8, \Gamma\bigr)$, where $\Gamma \in \bigl|\I_{2\overline{p}+2\overline{z}}(7)\bigr|$ is a curve such that $\mbox{Sing}(\Gamma)=\{z_1, \ldots, z_8\}$.
\end{definition}
We denote by $\mu\colon \mathcal{Y} \dashrightarrow \cS_{3,6}^+$ the forgetful morphism $\bigl(D, \eta, \overline{p}, z_1+\cdots+z_8, \Gamma\bigr) \mapsto [D, \eta]$.

\begin{proposition}\label{prop:Y-bir}
The morphism $\mathcal{Y}\dashrightarrow \P_{3,6}$ is birationally a $\PP^2$-bundle. 
\end{proposition}
\begin{proof}
Having fixed a general element $\bigl(D, \eta, \overline{p}, z_1+\cdots+z_8\bigr)$, from (\ref{eq:10}) the linear system $\bigl|\I_{2\overline{p}+2\overline{z}}(7)\bigr|$ is $10$-dimensional. Requiring that $z_i\in \mbox{sing}(\Gamma)$ is a linear condition for $i=1,\ldots, 8$ and via a local calculation is straightforward to see that these $8$ conditions are linearly independent. It follows that $\mathcal{Y}\rightarrow \P_{3,6}$ is birational to a $\PP^2$-bundle.
\end{proof}

Summarizing this construction, one has the following diagram,
\begin{equation}\label{eq:diagram7}
\xymatrix@C=0.3em{
  & \mathcal{Y}=\Bigl\{(D,\eta, \overline{p}, z_1+\cdots+z_8, \Gamma): D\cdot \Gamma=2\bigl(\sum_{i=1}^6 p_i+\sum_{j=1}^8 z_j\bigr)\Bigr\}  
   \ar[dl]_{\mu} \ar[dr]^{\chi} & & \\
   \cS_{3,6}^+  &    & \hs_7^- \ar[rr] & &  & \SS_7^-       
                 }
\end{equation}
where $\chi$ is given by $\bigl(D,\overline{p},\eta, z_1+\cdots+z_8,\ \Gamma\bigr)\mapsto \bigl(C, \OO_C(\nu^{-1}(p_1)+\cdots+\nu^{-1}(p_6)), \omega_C(-1)\bigr)$, with $\nu\colon C\rightarrow \Gamma$ being the normalization map.

%We summarize this construction, providing a direct relation between $\SS_{3,6}^+$ and $\SS_7^-$.

\begin{theorem}\label{thm:spin7}
The Hurwitz spin stack  $\hs_7^-$ is birational to a double projective bundle over  $\cS_{3,6}^+$.
\end{theorem}

\vskip 4pt
A variant of this construction  provides a birational correspondence between $\SS_8^-$ and $\SS_{4,7}^+$.


\begin{thebibliography}{aaaaaa}
\bibitem[AB]{AB} D. Agostini and I. Barros, {\em{Pencils on surfaces with normal crossings and the Kodaira dimension of
$\cM_{g,n}$}}, Forum Math. Sigma \textbf{9} (2021), Paper No. 31.
\bibitem[AC]{AC} E. Arbarello and M. Cornalba, {\em{Calculating cohomology groups of moduli spaces of
curves via algebraic geometry}}, Publications math. de l’I.H.\'{E}.S.,  \textbf{88} (1998),  97--127.
\bibitem[ACGH]{ACGH}
  E.~Arbarello, M.~Cornalba, P.\,A.~Griffiths and J.~Harris, {\emph{Geometry of algebraic curves, Vol.~I}}, Grundlehren math.\ Wiss.\ vol.~267, Springer-Verlag, New York, 1985.
\bibitem[Ba]{B} I. Barros, {\em{Uniruledness of strata of holomorphic differentials in small
genus}}, Advances in  Math. \textbf{333} (2018), 670--693.
\bibitem[BF]{BF} G. Bini and C. Fontanari, {\em{Moduli of curves and spin structures via algebraic geometry}}, Transactions American Math. Soc. \textbf{358} (2006), 3207--3217. 
\bibitem[Bu]{Bu} A. Bud, {\em{Maximal gonality on strata of differentials and uniruledness
of strata in low genus}}, Bulletin London Math. Soc. \textbf{53} (2021), 1627--1635.
\bibitem[BM]{BM} I. Barros and S. Mullane, {\em{Two moduli spaces of Calabi-Yau type}}, International Math. Res. Notices \textbf{20} (2021), 15833--15899.
\bibitem[CL]{CL} S. Canning and H. Larson, {\em{On the Chow and cohomology groups of moduli spaces of stable curves}}, arXiv:2208.02357,  Journal of the European Math. Soc. (2026).
\bibitem[CLP]{CLP} S. Canning, H. Larson and S. Payne, {\em{The eleventh cohomology group of $\mm_{g,n}$}}, Forum of Mathematics Sigma \textbf{11} (2023), 1--18.    
\bibitem[CMP]{CMP} L. Caporaso, M. Melo and M. Pacini, {\em{Tropicalizing the moduli space of spin curves}}, Selecta Math. \textbf{26} (2020), article 16.  
%\bibitem[CEFS]{CEFS} A. Chiodo, D. Eisenbud, G. Farkas and F.-O. Schreyer, {\em{Syzygies of torsion bundles and the geometry of the level $\ell$ modular varieties over $\overline{\mathcal{M}}_g$}}, Inventiones Math. \textbf{194} (2013), 73--118.
\bibitem[CCM]{CCM} D. Chen, M. Costantini, and M. M\"oller,  {\em{On the Kodaira dimension
of moduli spaces of Abelian differentials}}, Cambridge Journal of Math \textbf{12} (2024), 623–-752.
\bibitem[Cor]{Cor} M. Cornalba, {\em{Moduli of curves and theta-characteristics}},
in: Lectures on Riemann surfaces (Trieste, 1987), 560--589.
\bibitem[Dol]{Dol} I. Dolgachev, {\em{Classical algebraic geometry: a modern view}}, Cambridge University Press, 2012. 
\bibitem[DW]{DW} R. Donagi and E. Witten, {\em{Super Atiyah classes and obstructions to splitting of supermoduli space}}, Pure and Applied Math. Quarterly  \textbf{9} (2013), 739--788.
\bibitem[EH]{EH} D. Eisenbud and J. Harris, {\em{The Kodaira
dimension of the moduli space of curves of genus $\geq 23$}},
Inventiones Math. \textbf{90} (1987), 359--387.
\bibitem[FP]{FP} C. Faber and R. Pandharipande, {\em{Tautological and non-tautological cohomology of the moduli space of curves}}, Handbook of Mduli vol. 1, 293--330, Intl. Press 2013. 
\bibitem[F1]{F1} G. Farkas, {\em{Koszul divisors on moduli spaces of
curves}}, American Journal of Math. \textbf{131} (2009),
819--869.
\bibitem[F2]{F2} G. Farkas, {\em{The birational type of the moduli space of even spin curves}}, Advances in Math. \textbf{223} (2010), 433--443.
\bibitem[F3]{F3} G. Farkas, {\em{Brill-Noether geometry on moduli spaces of spin curves}}, in: Classification of Algebraic Varieties, EMS Series of Congress Reports (2011), 259--278.
\bibitem[FJP]{FJP} G. Farkas, D. Jensen and S. Payne, {\em{The Kodaira dimension of $\mm_{22}$ and $\mm_{23}$}}, Cambridge Journal of Math. \textbf{13} (2025), 431--607.
\bibitem[FP]{FP} G. Farkas and M. Popa, {\em{Effective divisors on
$\mm_g$, curves on $K3$ surfaces and the Slope Conjecture}}, Journal of Algebraic Geometry \textbf{14} (2005), 151--174.    
%\bibitem[FL]{FL} G. Farkas and K. Ludwig, {\em{The Kodaira dimension of the moduli
%space of Prym varieties}}, Journal of the European Mathematical Society \textbf{12} (2010), 755--795.
\bibitem[FV1]{FV1} G. Farkas and A. Verra, {\em{Moduli of theta-characteristics via Nikulin surfaces}}, Mathematische Annalen \textbf{354} (2012), 465--496.
\bibitem[FV2]{FV4} G. Farkas and A. Verra, {\em{The classification of universal Jacobians over the moduli space of curves}} 
Commentarii Mathematici Helvetici \textbf{88} (2013), 587--611.
\bibitem[FV3]{FV2} G. Farkas and A. Verra, {\em{The geometry of the moduli space of odd spin curves}}, Annals of Mathem. \textbf{180} (2014), 927--970.
\bibitem[FV4]{FV3} G. Farkas and A. Verra, {\em{The uniruledness of the Prym moduli space of genus $9$}}, Advances in Math. \textbf{448} (2024), Paper 109678.
\bibitem[GHS]{GHS} T. Graber, J. Harris and J. Starr, {\em{Families of rationally connected varieties}}, Journal Amer. Math. Soc.  \textbf{16} (2003), 57--67.    
\bibitem[HM]{HM} J. Harris and D. Mumford, {\em{On the Kodaira
dimension of $\mm_g$}}, Inventiones Math. \textbf{67} (1982), 23--88.  
\bibitem[KZ]{KZ} M. Kontsevich and A. Zorich, {\em{Connected components of the moduli spaces of Abelian differentials with prescribed singularities}}, Inventiones Math. \textbf{153} (2003), 631--678.
\bibitem[Kr]{Kr} S. Krug, {\em{Geometry of moduli spaces of spin and prym curves of small genus}}, PhD Thesis, Leibniz Universit\"at Hannover 2012, \ available at 
https://repo.uni-hannover.de/handle/123456789/7945.    
\bibitem[Log]{Log} A. Logan, {\em{The Kodaira dimension of moduli spaces of curves
with marked points}}, American Journal of Math. \textbf{125} (2003), 105--138.  
\bibitem[Lu]{Lu} K. Ludwig, {\em{On the geometry of the moduli space of spin
curves}},  Journal of Algebraic Geometry \textbf{19} (2010), 133--171. 
\bibitem[M1]{M1} S. Mukai, {\em{Curves and Grassmannians}}, in: Algebraic Geometry and Related Topics, eds. J.-H. Yang, Y. Namikawa, K. Ueno, 19--40, 1992   International Press.
\bibitem[M2]{M3} S. Mukai, {\em{Curves and symmetric spaces I}}, American Journal of Mathematics \textbf{117} (1995), 1627--1644.    
\bibitem[M2]{M2} S. Mukai, {\em{Curves and symmetric spaces II}}, Annals of Math. \textbf{172} (2010), 1539–-1558.   
\bibitem[Mu]{Mu} D. Mumford, {\em{Theta characteristics of an algebraic curve}}, Annales Scient. \'Ecole Norm. Sup. \textbf{4} (1971),
181--192.    



\end{thebibliography}
\end{document}